\documentclass[11pt]{amsart}

\usepackage{subfiles}
\usepackage{euscript}
\usepackage{tabu}
\usepackage[margin=1in]{geometry} 		             		
\usepackage{graphicx}
\usepackage{faktor}
\usepackage{xcolor}
\usepackage{amsmath,amsthm,amssymb}
\usepackage{mathabx}
\usepackage{mathrsfs}
\usepackage{tikz}
\usepackage{tikz-cd}
\usepackage{accents}
\usepackage{upgreek}
\usepackage{enumitem}
\usepackage{bm}
\usepackage{mathtools}
\usepackage[all]{xy}
\usepackage{caption}

\usepackage{url}
\usepackage{float}
\usepackage{todonotes}
\usepackage{bbm}
\usepackage{colonequals}
\usepackage{longtable}

\usepackage[full]{textcomp}
\usepackage[cal=cm]{mathalfa}
\usepackage{xparse}
\usepackage{comment}
\usepackage{cite}

\tikzset{
  commutative diagrams/.cd, 
  arrow style=tikz, 
  diagrams={>=stealth}
}

\theoremstyle{definition}

\makeatletter
\def\@tocline#1#2#3#4#5#6#7{\relax
  \ifnum #1>\c@tocdepth 
  \else
    \par \addpenalty\@secpenalty\addvspace{#2}%
    \begingroup \hyphenpenalty\@M
    \@ifempty{#4}{%
      \@tempdima\csname r@tocindent\number#1\endcsname\relax
    }{%
      \@tempdima#4\relax
    }%
    \parindent\z@ \leftskip#3\relax \advance\leftskip\@tempdima\relax
    \rightskip\@pnumwidth plus4em \parfillskip-\@pnumwidth
    #5\leavevmode\hskip-\@tempdima
      \ifcase #1
       \or\or \hskip 1em \or \hskip 2em \else \hskip 3em \fi%
      #6\nobreak\relax
    \dotfill\hbox to\@pnumwidth{\@tocpagenum{#7}}\par
    \nobreak
    \endgroup
  \fi}
\makeatother

\makeatletter
\DeclareRobustCommand{\cev}[1]{%
  \mathpalette\do@cev{#1}%
}
\newcommand{\do@cev}[2]{%
  \fix@cev{#1}{+}%
  \reflectbox{$\m@th#1\vec{\reflectbox{$\fix@cev{#1}{-}\m@th#1#2\fix@cev{#1}{+}$}}$}%
  \fix@cev{#1}{-}%
}
\newcommand{\fix@cev}[2]{%
  \ifx#1\displaystyle
    \mkern#23mu
  \else
    \ifx#1\textstyle
      \mkern#23mu
    \else
      \ifx#1\scriptstyle
        \mkern#22mu
      \else
        \mkern#22mu
      \fi
    \fi
  \fi
}

\makeatother

\usetikzlibrary{calc}
\usetikzlibrary{fadings}
\usetikzlibrary{decorations.pathmorphing}
\usetikzlibrary{decorations.pathreplacing}
\usetikzlibrary{shapes}

\newcommand{\CoreLevelFigure}{%
\begin{tikzpicture}[line width=.8pt,line cap=round,
 every node/.style={font=\small}]
\path[use as bounding box] (-3,-1.65) rectangle (3,1.7);
\draw (0,0) ellipse (1.12 and .72);
\draw (-.48,.13) .. controls (-.26,-.16) and (.26,-.16) .. (.48,.13);
\draw (-.37,.02) .. controls (-.16,.19) and (.16,.19) .. (.37,.02);
\node at (0,-.40) {$E$};
\draw[domain=-1:1,samples=50] plot ({-1.12-1.5*\x*\x},\x);
\draw[domain=-1:1,samples=50] plot ({1.12+1.5*\x*\x},\x);
\fill ({-1.12-1.5*.7*.7},.7) circle (1.5pt) node[above left] {$1$};
\fill ({-1.12-1.5*.7*.7},-.7) circle (1.5pt) node[below left] {$2$};
\fill ({1.12+1.5*.7*.7},.7) circle (1.5pt) node[above right] {$3$};
\fill ({1.12+1.5*.7*.7},-.7) circle (1.5pt) node[below right] {$4$};
\fill (0,.72) circle (1.5pt) node[above=3pt] {$5$};
\draw[fill=white] (-1.12,0) circle (2.2pt);
\draw[fill=white] (1.12,0) circle (2.2pt);
\node at (-2.1,0) {$R_1$};
\node at (2.1,0) {$R_2$};
\node at (0,1.48) {Core level $P$};
\node at (0,-1.48) {$P=\{\{1,2\},\{3,4\},\{5\}\}$};
\end{tikzpicture}%
}
\newcommand{\SingularityLevelFigure}{%
\begin{tikzpicture}[line width=.8pt,line cap=round,
 every node/.style={font=\small}]
\path[use as bounding box] (-3,-1.65) rectangle (3,1.7);
\draw (-2.4,-.85) -- (2.4,.85);
\draw (-2.4,.85) -- (2.4,-.85);
\draw (0,-1.05) -- (0,1.05);
\fill (-1.7,{-.85*1.7/2.4}) circle (1.5pt) node[below=3pt] {$1$};
\fill (1.7,{.85*1.7/2.4}) circle (1.5pt) node[above=3pt] {$2$};
\fill (-1.7,{.85*1.7/2.4}) circle (1.5pt) node[above=3pt] {$3$};
\fill (1.7,{-.85*1.7/2.4}) circle (1.5pt) node[below=3pt] {$4$};
\fill (0,.85) circle (1.5pt) node[right=4pt] {$5$};
\fill (-.055,-.055) rectangle (.055,.055);
\node at (.62,0) {$q$};
\node[right] at (2.4,.85) {$R_1$};
\node[right] at (2.4,-.85) {$R_2$};
\node[left] at (0,1.05) {$R_3$};
\node at (0,1.48) {Singularity level $P$};
\node at (0,-1.48) {$q$: elliptic triple point};
\end{tikzpicture}%
}
\newcommand{\BothLevelsFigure}{%
\begin{tikzpicture}[line width=.8pt,line cap=round,line join=round,
 every node/.style={font=\small}]
\path[use as bounding box] (-3,-2.2) rectangle (3,2.05);
\draw (-2.1,{.35-2.1/1.65}) -- (1.55,{.35+1.55/1.65});
\draw (-1.55,{.35+1.55/1.65}) -- (2.1,{.35-2.1/1.65});
\draw (0,-.9) -- (0,1.55);
\node[right] at (1.55,{.35+1.55/1.65}) {$E_1$};
\node[left] at (-1.55,{.35+1.55/1.65}) {$E_2$};
\node[left] at (0,1.55) {$E_3$};
\draw[domain=-1:.8,samples=50] plot ({-1.65-.5*\x*\x},{\x-.65});
\draw[domain=-1:.8,samples=50] plot ({1.65+.5*\x*\x},{\x-.65});
\draw[fill=white] (-1.65,-.65) circle (2.2pt);
\draw[fill=white] (1.65,-.65) circle (2.2pt);
\fill ({-1.65-.5*.60*.60},{.60-.65}) circle (1.5pt) node[left=3pt] {$1$};
\fill ({-1.65-.5*.75*.75},{-.75-.65}) circle (1.5pt) node[left=3pt] {$2$};
\fill ({1.65+.5*.60*.60},{.60-.65}) circle (1.5pt) node[right=3pt] {$3$};
\fill ({1.65+.5*.75*.75},{-.75-.65}) circle (1.5pt) node[right=3pt] {$4$};
\fill (0,1.13) circle (1.5pt) node[right=4pt] {$5$};
\fill (-.055,.295) rectangle (.055,.405);
\node at (.50,.35) {$q$};
\node[left] at (-2.2,-.65) {$R_1$};
\node[right] at (2.2,-.65) {$R_2$};
\node at (0,1.94) {Core level = singularity level = $P$};
\node at (0,-2.02) {$P=\{\{1,2\},\{3,4\},\{5\}\}$};
\end{tikzpicture}%
}

\newcounter{marginnote}
\DeclareMathAlphabet{\mathpzc}{OT1}{pzc}{m}{it}

\definecolor{sebdarkgreen}{rgb}{0.019,0.317,0.149}
\definecolor{sebgreen}{RGB}{36,157,55}
\definecolor{seblightgreen}{rgb}{0.784,0.952,0.780}
\definecolor{sebblue}{RGB}{0,123,194}
\definecolor{seblightblue}{RGB}{194,242,255}

\usepackage[pagebackref]{hyperref}
\hypersetup{
  colorlinks   = true,          
  urlcolor     = blue,          
  linkcolor    = purple,          
  citecolor   = sebblue             
}
\usepackage{cleveref}
\theoremstyle{theorem}
\newtheorem{theorem}{Theorem}[section]
\newtheorem*{theorem*}{Theorem}

\newtheorem*{conjecture*}{Conjecture}
\newtheorem{corollary}[theorem]{Corollary}
\newtheorem{lemma}[theorem]{Lemma}
\newtheorem{proposition}[theorem]{Proposition}
\theoremstyle{remark}
\newtheorem{remark}[theorem]{Remark}
\theoremstyle{definition}

\newtheorem*{runningexample*}{Running example}

\newtheorem*{aside*}{Aside}

\newtheorem{definition}[theorem]{Definition}

\newtheorem{notation}[theorem]{Notation}
\newtheorem{proposition-definition}[theorem]{Definition-Proposition}
\newtheorem{question}[theorem]{Question}

  \theoremstyle{theorem}

  \theoremstyle{theorem}

\DeclareMathOperator{\Pic}{Pic}

\DeclareMathOperator{\core}{core}

\DeclareMathOperator{\sing}{sing}

\newcommand{\Tail}[1]{\mathbf{Core}_{#1}}
\newcommand{\oTail}[1]{\mathbf{Core}_{#1}}
\newcommand{\cTail}[1]{\overline{\mathbf{Core}}_{#1}}

\newcommand{\Ell}[1]{\mathbf{Sing}_{#1}}
\newcommand{\oEll}[1]{\mathbf{Sing}_{#1}}
\newcommand{\cEll}[1]{\overline{\mathbf{Sing}}_{#1}}

\newcommand{\Min}[1]{\mathbf{Min}_{#1}}

\newcommand{\Atom}[1]{\mathbf{Atom}_{#1}}

\newcommand{\RR}{\mathbb{R}}

\newcommand{\Ycal}{\mathcal{Y}}

\newcommand{\Gm}{\mathbb{G}_{\rm{m}}}

\newcommand{\bcd}{\begin{center}\begin{tikzcd}}
\newcommand{\ecd}{\end{tikzcd}\end{center}}

\newcommand{\Aaff}{\mathbb{A}}

\newcommand{\PP}{\mathbb{P}}
\newcommand{\OO}{\mathcal{O}}

\newcommand{\ZZ}{\mathbb{Z}}

\newcommand{\dvr}{\Delta}

\newcommand{\Lcal}{\mathcal{L}}
\newcommand{\Ocal}{\mathcal{O}}
\newcommand{\Mcal}{\mathcal{M}}
\newcommand{\Pcal}{\mathcal{P}}

\newcommand{\Hcal}{\mathcal{H}}

\newcommand{\Dcal}{\mathcal{D}}
\newcommand{\Xcal}{\mathcal{X}}

\newcommand{\Gcal}{\mathcal{G}}
\newcommand{\Ccal}{\mathcal{C}}
\newcommand{\Ucal}{\mathcal{U}}

\newcommand{\Mrad}{\oM_{1,n}^{\text{rad}}}

\newcommand{\Gun}{\Gcal_{1,n}}

\newcommand{\calO}{\mathcal{O}}

\newcommand{\Spec}{\operatorname{Spec}}

\NewDocumentCommand{\compatibilitydatum}{m m m m m m O{} O{} O{}}{
\begin{equation*} \begin{tikzcd}[ampersand replacement=\&]
  \: \arrow{r} \& {#1} \arrow{r} \arrow{d}{#7} \& {#2} \arrow{r} \arrow{d}{#8} \& {#3} \arrow{r}{[1]} \arrow{d}{#9} \& \: \\
  \: \arrow{r} \& {#4} \arrow{r} \& {#5} \arrow{r} \& {#6} \arrow{r} \& \:
\end{tikzcd} \end{equation*}}

\NewDocumentCommand{\commutingsquare}{m m m m o O{} O{} O{} O{}}{
\begin{equation}\begin{tikzcd}[ampersand replacement=\&] \label{#5}
  #1 \arrow{r}{#6} \arrow{d}{#7} \& #2 \arrow{d}{#8} \\
  #3 \arrow{r}{#9} \& #4
\end{tikzcd}\IfValueTF{#5}{\label{#5}}{} \end{equation}}

\NewDocumentCommand{\cartesiansquare}{m m m m O{} O{} O{} O{}}{
\begin{equation*}\begin{tikzcd}[ampersand replacement=\&]
  #1 \arrow{r}{#5} \arrow{d}{#6} \arrow[dr, phantom, "\square"] \& #2 \arrow{d}{#7} \\
  #3 \arrow{r}{#8} \& #4
\end{tikzcd} \end{equation*}}

\NewDocumentCommand{\cartesiansquarelabel}{m m m m m O{} O{} O{} O{}}{
\begin{tikzcd}[ampersand replacement=\&]
  #1 \arrow{r}{#6} \arrow{d}{#7} \arrow[dr, phantom, "\square"] \& #2 \arrow{d}{#8} \\
  #3 \arrow{r}{#9} \& #4
\end{tikzcd}\IfValueTF{#5}{\label{#5}}{}
}

\NewDocumentCommand{\triangleofspaces}{m m m O{} O{} O{}}{
\begin{tikzcd} [ampersand replacement=\&]
#1 \arrow{r}{#4} \arrow[bend right]{rr}{#5} \& #2 \arrow{r}{#6} \& #3
\end{tikzcd}}

\newcommand{\on}{\operatorname}
\newcommand{\oM}{\overline{\mathcal M}}
\newcommand{\cM}{\overline{\mathbb M}_{1,n}}
\newcommand{\cG}{\mathbb G_{1,n}}

\newcommand{\oC}{\overline{C}}

\newcommand{\disc}[1]{{#1}^{{\rm disc}}}
\usepackage{hyphenat}
\begin{document}
 
\title[PGMS of Gorenstein curves of genus one]{The unbearable projectivity of being a good moduli space\\ \small{of Gorenstein curves of genus one with non-colliding markings}}

\author{Luca Battistella}
\address{Dipartimento di Matematica, Universit\`a di Bologna, Italy}
\email{luca.battistella2@unibo.it}

\author{Andrea Di Lorenzo}
\address{Dipartimento di Matematica, Universit\`a di Pisa, Italy}
\email{andrea.dilorenzo@unipi.it}

\author{Michele Pernice}
\address{Department of Mathematics, University of Washington, Seattle, WA, USA}
\email{mpernice@uw.edu}

\begin{abstract}
We classify all open substacks admitting a \emph{proper good moduli space} inside the stack $\Gun$ of log-canonically polarized Gorenstein curves of genus one with $n$ distinct marked points: they are exactly the stacks $\oM_{1,n}(c)$ defined by Bozlee--Kuo--Neff. Restricting to the case where $c$ is \emph{pure of level $m$}, we then give a numerical criterion for the \emph{projectivity} of the good moduli spaces: these turn out to be asymptotically almost never projective as $n$ tends to infinity. 
\end{abstract}

\maketitle
\setcounter{tocdepth}{1}
\tableofcontents

\section{Introduction}
It is often said that moduli theory is partly art, and partly science. The art here is taken from Bozlee--Kuo--Neff, Franzen, Kundera, Shakespeare, and Smyth. 
As for the science: the main contribution of this paper is an application of the new theory of stacks due to J. Alper, D. Halpern-Leistner, and J. Heinloth to the moduli stack of pointed curves of genus one that highlights the close connection with logarithmic geometry.

Let $\mathbb{M}_{1,n}$ denote the moduli space of smooth genus one curves with $n$ distinct markings, and let $\cM$ denote the Deligne--Mumford compactification by stable curves. The birational geometry of the latter is still largely mysterious. Informally speaking, our paper revolves around the following:
\begin{question}[informal]
    What are the projective birational models of $\cM$?
\end{question}
We will not answer this question, which seems out of reach, especially because we know that $\cM$ has infinitely many extremal effective divisors for $n\geq3$ \cite{ChenCoskun}, and in particular it is not a Mori dream space. On the other hand, we will classify the models that arise as \emph{good moduli spaces} for some algebraic stacks parametrising Gorenstein curves with distinct, smooth markings, and we will study their projectivity.

To state the problem formally, consider the moduli stack $\Gun$ parametrising Gorenstein, reduced curves of arithmetic genus one, together with $n$ distinct sections of the smooth locus such that the log canonical bundle $\omega_C(p_1+\ldots+p_n)$ is ample. From a stack-theoretic viewpoint, this is a very mildly non-separated algebraic stack: every curve admits at most one worse-than-nodal singularity, and the automorphism groups are either finite (as for ordinary elliptic curves), or isomorphic to the multiplicative group. We then ask the following:
\begin{question}[formal]\label{question}
  What are the open substacks of $\Gun$ 
  admitting a proper good moduli space?
\end{question}
 The open substacks of $\Gun$ that are proper and Deligne--Mumford have been classified by S. Bozlee, B. Kuo, and A. Neff \cite{BKN}. They are parametrised by certain functions $c\colon\on{Part}([n])\to \{0,1\}$ specifying which singularities (resp. rational tail structures) are allowed. The same authors have introduced more general algebraic substacks, parametrised by functions $c\colon\on{Part}([n])\to [0,1]$ interpolating among the previous ones. Our first main result, which answers \Cref{question} and generalizes the Bozlee--Kuo--Neff classification, is the following.
 \begin{theorem}[\Cref{theo:classification}]
     The stacks $\oM_{1,n}(c)$ introduced by Bozlee, Kuo, and Neff admit proper good moduli spaces. Conversely, every open substack of $\Gcal_{1,n}$ admitting a proper good moduli space is of the form $\oM_{1,n}(c)$ for some BKN function $c\colon\on{Part}(n)\to [0,1]$.
 \end{theorem}
 Stacks admitting a proper good moduli space are characterised in \cite{AHLH} as the ones that are $\Theta$-complete, $S$-complete and universally closed. Our strategy consists in classifying the open substacks which satisfy these conditions. Our main contribution is the realisation that these conditions are very natural from the perspective of logarithmic geometry, upon which the definition of the $\oM_{1,n}(c)$ is based.
 
 We then study the projectivity of the good moduli spaces $\cM(c)$. Our starting point is the maximal open substack $\Gun(m)$ admitting a proper good moduli space and containing both of Smyth's stacks $\oM_{1,n}(m-1)$ and $\oM_{1,n}(m)$: it parametrises curves with at worst $m$-fold elliptic points and core level at least $m$. We prove that its good moduli space $\cG(m)$ is projective. All the stacks $\oM_{1,n}(c)\subset\Gun(m)$ with projective moduli space are semistable loci for line bundles $\mathcal{L}$ relative to $\cG(m)$: this imposes conditions on $\mathcal{L}$, eventually leading to our second main result. 
\begin{theorem}[\Cref{thm:condition for proj}]\label{thm:projectivity intro}
    The good moduli space of a stack $\oM_{1,n}(c) \subset \Gun(m)$ is projective if and only if there exist real values $a$, $\{b_S\}_{S\subset [n]}$ such that for every partition of the form $P=\{P_1,\ldots, P_m\}$ we have
    \begin{equation*}
    \left\{
    \begin{matrix}
        a - \sum_{\lvert P_i \rvert \geq 2} b_{P_i}<0 & &\text{ if }c(P)=1,\\
        a - \sum_{\lvert P_i \rvert \geq 2} b_{P_i}=0 & &\text{ if }0<c(P)<1, \\
        a - \sum_{\lvert P_i \rvert \geq 2} b_{P_i}>0 & &\text{ if }c(P)= 0.
    \end{matrix}
    \right.
    \end{equation*}
\end{theorem}
This is an explicit description of the semistability/$\mathbb Q$-factorialisation fan of $\cG(m)$, see D. Gori's PhD thesis \cite{GoriThesis}. We use our theorem to deduce the following.
\begin{corollary}[\Cref{cor:number of non-prog gms}]
   The following hold true:
    \begin{enumerate}
        \item for $m=1,2,n-1,n$ all the stacks $\oM_{1,n}(c) \subset \mathcal{G}_{1,n}(m)$ have a projective good moduli space;
        \item for other values of $m$, there are stacks $\oM_{1,n}(c)$, both Deligne-Mumford and not, that have non-projective moduli spaces.
    \end{enumerate}
    More precisely, in the second case, there are at least $\sum_{i=d+1}^{S(n,m)} \binom{S(n,m)}{i}$ Deligne-Mumford stacks with a non-projective moduli space, where $d= \sum_{i=2}^{n-m+1} \binom{n}{i}$ and $S(n,m)$ is the Stirling number of the second kind. 
\end{corollary}
The case $m\leq2$ encompasses moduli spaces of curves with only nodes, cusps, and tacnodes; projectivity in this case seems to be a more general (any $g\geq1$) phenomenon, analysed in detail in the work of G. Codogni, L. Tasin, and F. Viviani \cite{CTV1,CTV2}.

The second part of the corollary above implies that the moduli stacks $\oM_{1,n}(c)\subset \Gun(m)$ whose good moduli spaces are projective become increasingly rare as $n$ grows: in particular for $m\geq 3$ their ratio is bounded above by $f(n)=\exp({n2^n-m^n})$.

We also give an exotic example of a 
stack $\oM_{1,n}(c)$ \emph{not} contained in any $\Gun(m)$, whose moduli space has trivial Picard group (in particular, it cannot be projective). It suggests perhaps that the behavior of the moduli spaces not covered by \Cref{thm:projectivity intro} can be even wilder.

\subsection*{Did you guys use AI?\nopunct} Conversations with AI (Gemini, ChatGPT) were helpful in two instances: for studying the number of regions created by a hyperplane arrangement in a real vector space, relevant for \Cref{lm:Pic G(m)} and \Cref{cor:asymptotics}, and for proofreading.
 
\subsection*{Acknowledgements} We are grateful to Sebastian Bozlee, Francesca Carocci, Giovanni Inchiostro, Diane Maclagan, Dhruv Ranganathan, and Jonathan Wise for helpful conversations related to this paper. 

\subsection*{Notation and conventions} We let $n$ denote an integer $n\geq2$, and $m$ an integer $0\leq m\leq n$. We denote by $[n]$ the set $\{1,\ldots,n\}$, by ${[n]\choose k}$ the collection of all subsets of $[n]$ of fixed cardinality $k$, and by $[n]_a^b$ the collection of all subsets of $[n]$ of cardinality ranging between $a$ and $b$. We denote by $\Pi_n$ the collection of set partitions of $[n]$, and by $\Pi_{n,m}$ the collection of set partitions of $[n]$ into $m$ parts (or blocks). The cardinality of $\Pi_{n,m}$ is called the Stirling number of the second kind and is denoted by $S(n,m)$. We work throughout over an algebraically closed field of characteristic $0$.

\section{Preliminaries on good moduli spaces}

\subsection{Good moduli spaces for algebraic stacks}
The definition of a good moduli space for an Artin stack was given by Alper in his Ph.D. thesis \cite[Definition 4.1]{AlperGood}, generalising the notion of the coarse moduli space of a Deligne--Mumford stack by extending 
the ideas of Mumford's Geometric Invariant Theory (GIT).

\begin{definition}\label{def:gms}
    Let $\Xcal$ be an Artin stack (with separated, quasi-compact diagonal). A morphism to an algebraic space $\phi\colon\Xcal \to X$ is called a \emph{good moduli space} if:
    \begin{enumerate}[leftmargin=.5cm]
       \item $\phi$ is cohomologically affine, i.e. the pushforward of quasi-coherent sheaves along $\phi$ is exact, and
       \item the natural map $\OO_X\to\phi_*\OO_{\Xcal}$ is an isomorphism.
    \end{enumerate}
\end{definition}

Good moduli spaces enjoy several desirable properties, such as being surjective and universally closed, stability under pullback, and universality for maps to algebraic spaces.

While the Keel--Mori theorem identifies the conditions for a Deligne--Mumford stack to admit a coarse moduli space, Alper, Halpern-Leistner, and Heinloth have found criteria for an algebraic stack to admit a good moduli space \cite[Theorem A]{AHLH}. Interestingly, these are all valuative criteria. We recall the relevant definitions before stating their theorem.

\begin{definition}
    Let $\Theta$ denote the quotient stack $[\Aaff^1/\Gm]$, and let $\Theta_R$ denote the product $\Theta \times \Spec R$. If $R$ is local with residue field $k$, let $0$ denote the point $B\Gm \times \Spec k $ of $\Theta_R$. A morphism of locally Noetherian algebraic stacks $f\colon\Xcal \to\Ycal$ is called \emph{$\Theta$-complete} if for every discrete valuation ring (DVR) $R$, any commutative diagram
    \bcd
    \Theta_R\smallsetminus 0\ar[r]\ar[d,hook] & \Xcal\ar[d,"f"] \\
    \Theta_R\ar[r]\ar[ur,dashed] & \Ycal \\
    \ecd
    of solid arrows can be filled in uniquely.
\end{definition}
Observe that $\Theta$ has an open point $\circ=[\mathbb{A}^1\smallsetminus\{0\}/\Gm]$ and a closed stacky point $\bullet =[\Spec(k)/\Gm]$. Therefore when giving a map $f\colon\Theta \to \mathcal{X}$ we have in particular objects $x:\circ\to\mathcal{X}$ and $x_0:\bullet\to \mathcal{X}$, together with a specialisation of $x$ to $x_0$; we will sometimes denote the map $f$ as $x\rightsquigarrow x_0$.
\begin{definition}
    Let $R$ be a DVR with a choice of uniformiser $\pi$. Consider the quotient:
    \[\overline{ST}_R=[\Spec(R[s,t]/(st-\pi))/\Gm],\]
    where $\Gm$ acts with weight $1$ on $s$ and $-1$ on $t$. We denote by $0$ the closed point given by $s=t=0$. A morphism of locally Noetherian algebraic stacks $f\colon\Xcal \to\Ycal$ is called \emph{S-complete} if for every DVR, any commutative diagram
    \bcd
    \overline{ST}_R\smallsetminus 0\ar[r]\ar[d,hook] & \Xcal\ar[d,"f"] \\
    \overline{ST}_R\ar[r]\ar[ur,dashed] & \Ycal \\
    \ecd
    of solid arrows can be filled in uniquely.
\end{definition}

We are ready to state \cite[Theorem A]{AHLH}.
\begin{theorem}\label{gms_criterion}
    Let $\Xcal$ be an algebraic stack of finite presentation, with affine stabilisers and separated diagonal over a Noetherian algebraic space $S$ of characteristic $0$.
Then $\Xcal$ admits a separated good moduli space $X$ if and only if $\Xcal$ is $\Theta$-complete and S-complete.
Moreover, $X$ is proper if and only if $\Xcal$ satisfies the existence part of the valuative criterion for properness.
\end{theorem}

We will say "is universally closed" for "satisfies the existence part of the valuative criterion for properness". 

\begin{remark}
    The stacks $\Theta_R$ and $\overline{ST}_R$ are simple examples of Olsson cones \cite{BCW}. The relationship between the geometry of algebraic stacks and logarithmic geometry is one of the themes that we wish to highlight throughout this work.
\end{remark}

Next, we are going to collect some useful technical results regarding the two valuative criteria.

\begin{lemma}\label{lem:map-in-the-middle}
    Let $f:\Xcal \rightarrow \Ycal$ and $g:\Ycal \rightarrow S$ be morphisms of locally noetherian algebraic stacks. If $g\circ f$ is $S$-complete (respectively $\Theta$-complete) and $g$ has affine diagonal, then $f$ is $S$-complete (respectively $\Theta$-complete). 
\end{lemma}
\begin{proof}
    This follows from the following general observation: let $\Pcal$ be a property of morphisms that is stable under composition and base change. If $g\circ f$ and $\Delta_g$ both have $\Pcal$, then so does $f$ by the following diagram:
    \bcd
& \Xcal\ar[r]\ar[dl] & \Xcal \times_S\Ycal\ar[dl]\ar[dr]\ar[r] & \Ycal\ar[dr] & \\
\Ycal\ar[r,"\Delta_g"] & \Ycal\times_S\Ycal & & \Xcal\ar[r,"g\circ f"] & S.\\
    \ecd
    Notice that affine morphisms are both $\Theta$- and $S$-complete (this follows from Propositions 3.21 and 3.44 of \cite{AHLH}).
\end{proof}

\begin{lemma}\label{lem:descending_closedness}
    Let $f\colon\Xcal\to\Ycal$ be a surjective morphism of stacks over $S$. If $\Xcal$ is universally closed over $S$, then so is $\Ycal$.
\end{lemma}
\begin{proof} See \cite[\href{https://stacks.math.columbia.edu/tag/03GN}{Lemma 03GN}]{stacks-project}.
\end{proof}

 \begin{lemma}\label{lem:descending_S-completeness}
    Let $\pi:\Xcal \rightarrow \Ycal$ be a universally closed, surjective morphism of locally noetherian algebraic stacks and suppose $\Ycal$ has affine diagonal. If $\Xcal$ is $S$-complete then $\Ycal$ is $S$-complete.
 \end{lemma}

\begin{proof} Let $f^{\circ}:\overline{ST}_R\smallsetminus 0 \rightarrow \Ycal$ and let $f_K:\Spec K \rightarrow \Ycal$ be its restriction to the generic point of $\overline{ST}_R\smallsetminus 0$ (corresponding to $\pi\neq 0$).  Since $\Ycal$ has affine diagonal, it is enough to find a filling after an arbitrary extension \cite[Proposition 3.17(2)]{AHLH}. Moreover $\pi$ is surjective, thus we can lift the morphism $f_K$ to a morphism $g_K:\Spec K \rightarrow \Xcal$ up to an extension of the fraction field $K$. Notice that $\overline{ST}_R \smallsetminus 0 =\Spec R \cup_{\Spec K} \Spec R$, with $K$ the fraction field of $R$, therefore the universal closedness of $\pi$ allows us to lift the morphism $f^{\circ}$ to a morphism $g^{\circ}:\overline{ST}_R\smallsetminus 0 \rightarrow \Xcal$ up to a further extension of the DVR. The $S$-completeness of $\Xcal$ provides us with a morphism $g:\overline{ST}_R\rightarrow \Xcal$ which restricts to $g^{\circ}$ on the open subset $\overline{ST}_R\smallsetminus 0$. The morphism $f:=\pi \circ g: \overline{ST}_R \rightarrow \Ycal$ yields the desired filling of $f^{\circ}$.
\end{proof}

An analogous statement (see \Cref{prop:descending-theta-complete} below) holds for $\Theta$-completeness, but its proof is a bit more subtle. 
For a field $K$, let $j\colon\Spec(K) \to \Theta_K$ be the open immersion. We say that a morphism $f\colon\mathcal{X} \to \mathcal{Y}$ verifies \emph{the valuative criterion of $\Theta$-surjectivity} \cite[Remark 3.26]{AHLH} if for any field $K$, any commutative diagram of solid arrows
\[
\begin{tikzcd}
\Spec K \arrow[d, "j"] \arrow[r, "x"] & \Xcal \arrow[d, "f"] \\
\Theta_K \arrow[r]  \ar[ru, dashed]          & \Ycal               
\end{tikzcd}
\]
can be filled in. Recall also that the $n$-th root stack of $\Theta_K$ along $B\Gm$ is the map $\Theta_K\to\Theta_K$ obtained by pulling back the map $B\Gm\to B\Gm$ induced by the homomorphism $\Gm\to\Gm,$ $x\mapsto x^n$.
\begin{lemma}\label{lem:theta-surjectivity}
    Let $f:\Xcal \rightarrow \Ycal$ be a universally closed morphism of finite type between locally noetherian locally reductive algebraic stacks over $\mathbb{Q}$ with affine diagonal. Then $f$ verifies the valuative criterion of $\Theta$-surjectivity up to 
    \begin{itemize}
        \item taking extensions of the field $K$,
        \item taking the $n$-th root stack of $\Theta_K$ along $B\Gm$.
    \end{itemize}
\end{lemma}
\begin{remark}
The main new feature of this setup when compared with \Cref{lem:descending_S-completeness} is precisely the fact that the statement holds up to taking root stacks. This is not unexpected: indeed, the universal $n$-th root stack $[n]:\Theta_K \rightarrow \Theta_K$ is universally closed but not $\Theta$-surjective. In particular, the same strategy adopted in the proof of  \Cref{lem:descending_S-completeness} cannot apply.
\end{remark}
\begin{proof}
Firstly, we can assume $K$ is algebraically closed since we want to prove the criterion up to extensions of the field. Consider the diagram
$$\begin{tikzcd}
\Spec K \arrow[d] \arrow[r,"x"] & \Xcal \arrow[d, "f"] \\
\Theta_K \arrow[r, "y\rightsquigarrow y_0"]           & \Ycal               
\end{tikzcd}$$
The universal closedness of $f$ implies that the closure of $x$ in $\Xcal\times_{\Ycal}\Theta_K$ intersects the special fiber over $B\mathbb{G}_{m,K} \subset \Theta_K$ non-trivially and call $x_0$ a closed point in the intersection. Assume for the moment that $\Xcal \times_{\Ycal} \Theta_K$ is locally reductive \cite[Definition 2.5]{AHLH}. As it is also of finite type over $\Theta_K$, there exists a morphism $\Theta_K \to \Xcal\times_{\Ycal}\Theta_K$ realizing the specialization $x \rightsquigarrow x_0$, see \cite[Proposition 3.24]{AHLH}. 
This induces a commutative diagram
$$
\begin{tikzcd}
\Theta_K  \arrow[d, "g"] \arrow[r, "x \rightsquigarrow x_0"] & \Xcal \arrow[d, "f"] \\
\Theta_K \arrow[r, "y \rightsquigarrow y_0"]                 & \Ycal               
\end{tikzcd}
$$
where $g$ is a surjective morphism which is an isomorphism on the generic fiber. Contrary to the case of spectra of discrete valuation rings, where a birational morphism is necessarily an isomorphism, $g$ need not be an isomorphism; nevertheless, we can describe it precisely. The Hom-stack $\underline{\rm Map}(\Theta, \Theta)$ is a disjoint union $\bigsqcup_{n \in \mathbb{Z}} \Theta^{[n]}$ where (see Section 2.4 of \cite{AHLH})
\begin{itemize}
    \item for $n\geq 0$ we have that $\Theta^{[n]}\simeq\Theta$ parametrizes morphisms induced by $t\mapsto \lambda t^n$,
    \item for $n<0$ we have that $\Theta^{[n]}\simeq B\Gm$: the corresponding morphism can be described as the composition $\Theta\rightarrow B\Gm\rightarrow B\Gm \hookrightarrow \Theta$, where the central morphism is the morphism of classifying stacks induced by the $n$-th power $\Gm \rightarrow \Gm$.
\end{itemize}
Since the second case does not give rise to a surjective morphism, there exists a positive integer $n$ such that $g$ can be identified with the universal $n$-th root stack $[n]:\Theta \rightarrow \Theta$, corresponding to the case $\lambda\neq 0$.

To complete the proof, we still need to resolve the initial technical issue, namely that $\Xcal \times_{\Ycal} \Theta_K$ has to be locally reductive. Since we are in characteristic zero, the \'etale local structure theorem (see \cite[Theorem 1.1]{Luna}) reduces the problem to showing that $\Xcal \times_{\Ycal} B \mathbb{G}_{m,K}$ has linearly reductive stabilizers at closed points.

First of all, we notice that any morphism
    $$  B \mathbb{G}_{m,K} \longrightarrow \Ycal$$
    is universally cohomologically affine, i.e. for every $T\rightarrow \Ycal$, the base change $T \times_{\Ycal}  B \mathbb{G}_{m,K}\rightarrow T$ is cohomologically affine (see \cite[Definition 3.1]{AlperGood}). Indeed, we have $T \times_{\Ycal}  B \mathbb{G}_{m,K} \simeq (T\times B \mathbb{G}_{m,K}) \times_{\Ycal \times \Ycal} \Ycal$; as the diagonal $\Ycal \to \Ycal\times\Ycal$ is affine, $T\times_\Ycal B \mathbb{G}_{m,K}\to T\times  B \mathbb{G}_{m,K}$ is also affine. The composition $T\times_\Ycal B \mathbb{G}_{m,K}\to T\times  B \mathbb{G}_{m,K} \to T$ is thus cohomologically affine if (and only if) the projection $T\times  B \mathbb{G}_{m,K}\to T$ is cohomologically affine, which follows from $ B \mathbb{G}_{m,K} \to \Spec(K)$ being cohomologically affine.

    To conclude, we will show that $\Xcal \times_{\Ycal}  B \mathbb{G}_{m,K}$ has linearly reductive stabilizers at closed points. Given $x$ a closed point of $\Xcal \times_{\Ycal}B \mathbb{G}_{m,K}$ and $x_0 \in \overline{\{x\}}$ a closed point in $\Xcal$, we can consider an \'etale local quotient presentation $\mathcal{W}=[\Spec(A)/\on{GL}_n] \rightarrow \Xcal$ around $x_0$, where $\mathcal{W}$ has affine good moduli space $q:\mathcal{W} \to W:=\Spec(A^{\on{GL}_n})$ \cite[Theorem 1.1]{AHR-local}. Since $p_1\colon \mathcal{W} \times_{\Ycal} B\Gm \rightarrow \mathcal{W}$ is cohomologically affine, the source too has an affine good moduli space given by $\Spec({q_*p_1}_* \Ocal_{\mathcal{W} \times_{\Ycal} B \mathbb{G}_{m,K}})$.
    
    Since the local quotient presentation is \'etale, $x$ is in the image of the local presentation and thus there exists a closed point $w$ of $\mathcal{W} \times_{\Ycal}B \mathbb{G}_{m,K}$ which maps to $x$. Moreover, $\mathcal{W} \times_{\Ycal}B \mathbb{G}_{m,K}$ has an affine good moduli space, thus $w$ has linearly reductive automorphism group and, because the morphism $G_w \rightarrow G_x$ is \'etale, we get that $G_x$ is also linearly reductive, since it differs from $G_w$ by a finite group and we are working in characteristic zero.
\end{proof}
\begin{remark}
    We want to stress that it is not surprising that the root stack construction appears in the valuative criterion. Indeed, even in the classical case with DVRs, this was already pointed out in \cite{BreVis}. Moreover, one does not really need all the technical hypotheses in the previous lemma, specifically the finite type condition or the characteristic assumption. The statement of the lemma can be improved but one would need the theory of topological moduli spaces, which will appear in a forthcoming paper by David Rydh. 
\end{remark}
\begin{proposition}\label{prop:descending-theta-complete}
     Let $\pi:\Xcal \rightarrow \Ycal$ be a universally closed, surjective morphism of finite type of locally noetherian locally reductive algebraic stacks over $\mathbb{Q}$ with affine diagonal. If $\Xcal$ is $\Theta$-complete then $\Ycal$ is $\Theta$-complete.
\end{proposition}

\begin{proof}
    Let $f:\Theta_R   \smallsetminus 0 \rightarrow \Ycal$ be any morphism where $R$ is a DVR. Notice that $\Theta_R \smallsetminus 0$ is the pushout of $\Theta_K$ and $\Spec R$ along $\Spec K$. Therefore we can lift the morphism $\Spec R \rightarrow \Ycal$ to a morphism $\Spec R \rightarrow \Xcal$, since the morphism is surjective and universally closed, up to extending the DVR. Moreover, we can apply \Cref{lem:theta-surjectivity} and lift the morphism $\Theta_K\rightarrow \Ycal$ to a morphism $\Theta_K \rightarrow \Xcal$ up to possibly taking further extensions of $K$ (and thus of the DVR) and/or a root stack construction. Thus we found a morphism $g:\Theta_R \smallsetminus 0\to \Xcal$ which lifts $f$ up to composing with the root stack of $\Theta_K$. Following the strategy of \Cref{lem:descending_S-completeness}, the $\Theta$-completeness of $\Xcal$ gives us the filling $\Theta_R \rightarrow \Xcal$ of $g$, and postcomposing with $\pi$ gives the filling of $f$ up to the root stack construction. 

    We have proved the following valuative criterion for $\Ycal$: for every discrete valuation ring $R$ and every diagram 
    $$
    \begin{tikzcd}
\Theta_R \smallsetminus 0 \arrow[rd, "f"] \arrow[d, hook] &       \\
\Theta_R                                             & \Ycal
\end{tikzcd}
$$ 
there exists a positive integer $n$ and a morphism $F:\Theta_R \rightarrow \Ycal$ such that the following diagram
$$ 
\begin{tikzcd}
\Theta_R \smallsetminus 0 \arrow[d, hook] \arrow[r, "{[n]}"]  & \Theta_R \smallsetminus 0 \arrow[rd, "f"] \arrow[d, hook] &       \\
\Theta_R \arrow[rr, "F"', bend right] \arrow[r, "{[n]}"] & \Theta_R                                             & \Ycal
\end{tikzcd}
$$
commutes, up to extensions of DVRs. Notice that, by restricting the upper horizontal morphism to the open $\Theta_K \subset \Theta_R \smallsetminus 0$, we get the following commutative diagram of solid arrows
$$
\begin{tikzcd}
\Spec K \arrow[d] \arrow[r, "{[f\vert_{\Theta_K}]}"] & {{\rm Map}(\Theta, \Ycal)} \arrow[d, "{[n]_*}"] \\
\Spec R \arrow[r, "{[F]}"]    \ar[ur, dashed, "{[F']}"]                       & {{\rm Map}(\Theta, \Ycal)} .                    
\end{tikzcd}
$$
where $[F]$ (respectively $[f\vert_{\Theta_K}]$) is the $R$-point representing $F\colon \Theta_R\to \mathcal{Y}$ (respectively the $K$-point representing $f\vert_{\Theta_K}\colon \Theta_K \to \mathcal{Y}$). The morphism $[n]_*\colon {\rm Map}(\Theta,\Ycal) \rightarrow {\rm Map}(\Theta,\Ycal)$ is induced by the universal $n$-th root stack $[n]: \Theta\rightarrow \Theta$ and it is a closed embedding \cite[Proposition 1.3.11]{instability}. Therefore, by the valuative criterion of properness, we can find, up to extending the DVR again, a dashed arrow as above which corresponds to a morphism $F':\Theta_R \rightarrow \Ycal$  making the following diagram commute:
$$ 
\begin{tikzcd}
\Theta_R \smallsetminus 0 \arrow[d, hook] \arrow[r, "{[n]}"]  & \Theta_R \smallsetminus 0 \arrow[rd, "f"] \arrow[d, hook] &       \\
\Theta_R \arrow[rr, "F"', bend right] \arrow[r, "{[n]}"] & \Theta_R           \arrow[r, "F'"]                                  & \Ycal.
\end{tikzcd}
$$
This concludes the proof.
\end{proof}
As often happens, it is useful to be able to check a local criterion by assuming that the image of the generic point of a DVR lands in a dense open substack. Below, we generalize a well-known criterion for Deligne--Mumford stacks to the setting of algebraic stacks. 

Given an algebraic stack, let us denote by ${\rm Irr}(X)$ the set of \emph{minimal} points, i.e. the generic points of its irreducible components. 

 \begin{definition}\label{def:gen univ cl}
     Let $\Xcal \to \Ycal$ be a morphism of locally noetherian algebraic stacks. We say that the morphism $f$ is \emph{generically closed} if for any discrete valuation ring $R$ with fraction field $K$, and  for any commutative diagram of solid arrows
    \[
    \begin{tikzcd}
    \Spec(K) \arrow[d, hook] \arrow[r, "x"] & \Xcal \arrow[d,"f"] \\
    \Spec(R) \arrow[r] \arrow[ru, dashed]   & \Ycal          
    \end{tikzcd}
    \]
    with $x \in {\rm Irr}(\vert \Xcal \vert)$, there exists, up to an extension of $R$, a dashed arrow that fills in the diagram.
 \end{definition}
 \begin{lemma}\label{lem:generic_univ_closed}
  Let $f:\Xcal \to \Ycal$ be an $S$-complete morphism between locally noetherian algebraic stacks which is of finite presentation and has affine diagonal. Suppose that $f$ is $\Theta$-complete and generically closed. Then $\Xcal \to \Ycal$ is universally closed.
 \end{lemma}
 \begin{proof}
    Since universal closedness is a property local on the target, we can assume that $\Ycal$ is a scheme. Notice that being generically closed is stable under base change along universally open morphisms.
    
    Thus, there exists a good moduli space $\Xcal \overset{\pi}{\to} X \to \Ycal$, with $g:X\to\Ycal$ separated (see \cite[Theorem A]{AHLH}). It is therefore enough to prove that $g$ is universally closed. Observe that since $\pi$ is surjective, every generic point of an irreducible component of $\Ycal$ is the image of a generic point of an irreducible component of $\Xcal$, therefore $g$ is generically closed. Indeed, given a diagram of solid arrows 
    \begin{equation} \begin{tikzcd}
    \Spec(K) \arrow[d, hook] \arrow[r]  & X \arrow[d] \\
    \Spec(R) \arrow[r]     & \Ycal,         
    \end{tikzcd}\end{equation}
    up to taking an extension of $R$, we can lift $\Spec(K)\to X$ to a morphism $\Spec(K)\to \Xcal$ which lands in ${\rm Irr}(\lvert \Xcal\rvert )$, and, by hypothesis, we can find a filling $\Spec(R)\to \Xcal$. Postcomposing with $\pi$, we get the desired filling.
    Finally, the separatedness of $X\to\Ycal$ ensures that generic closedness implies universal closedness in the algebraic space case \cite[Exercise 5.6.7]{Alper-book}.
\end{proof}
\begin{remark}
    In the previous Lemma, the converse is also true. Namely, universal closedness and $S$-completeness imply $\Theta$-completeness. We do not provide a proof since we do not make use of such a statement, but it is a consequence of the theory of topological moduli spaces developed in forthcoming work of David Rydh. 
\end{remark}

\section{Preliminaries on Gorenstein curves of genus one}

\begin{definition}
    We denote by $\Gcal_{1,n}$ the Artin stack of proper, reduced, Gorenstein curves $C$ of arithmetic genus one with $n$ distinct, smooth sections $p_1,\ldots,p_n$, such that the log canonical bundle $\omega_C(p_1+\ldots+p_n)$ is ample.
\end{definition}

Smyth \cite[Appendix A]{Smyth} classified all isolated Gorenstein singularities of genus one.
\begin{definition}
   Let $k$ be a field of characteristic different from $2,3$. A $k$-point $p$ of a curve $C$ is called an \emph{elliptic} $m$\emph{-fold point} if the analytic germ of $C$ at $p$ is one of the following:
    \[\hat{\calO}_{C,p}\simeq 
    \left\{ \begin{matrix}
       &k [\![ x,y ]\!]/(y^2-x^3) & m=1 &\text{ordinary cusp, } A_2 \\
       &k [\![  x,y ]\!]/(y^2-yx^2) & m=2 &\text{ordinary tacnode, }A_3 \\
       &k [\![  x,y ]\!]/(x^2y-xy^2) & m=3 &\text{planar triple point, }D_4 \\
       &k [\![  x_1,\ldots,x_{m-1} ]\!]/I_m & m\geq 4 &\text{$m$-general lines through the origin of }\Aaff^{m-1}
    \end{matrix}\right.
    \]
    where $I_m$ is the ideal generated by the binomials $x_ix_j-x_ix_h$ for all $i\neq j,h\in[m-1]$.
\end{definition}
\begin{definition}
  A Gorenstein curve of genus one is called \emph{minimal} if it does not contain any disconnecting node. We denote by $\Min{n}\subset\Gcal_{1,n}$ the open substack of minimal curves; it is the complement of the divisor of rational tails (with at least two markings).
 \end{definition}
\begin{remark}\label{rmk:minimal}
    Minimal curves admit a normal form, yielding explicit equations for $\Min{n}$ \cite{LekiliPolishchuk}: denoting by $[C_0]$ the point corresponding to the curve with an elliptic $n$-fold singularity, it turns out that $\Min{n}\smallsetminus\{[C_0]\}$  can always be identified with the projectivisation of the base of the miniversal deformation of the elliptic $n$-fold point \cite{Stevens}. For $n\leq 5$, this is a weighted projective stack, a fact that can be harnessed to describe the relationship between $\oM_{1,n}$ and Smyth's alternative compactifications as a sequence of weighted blow-ups and downs \cite{BDL4}. We wonder whether such a description could be pushed further to all $\oM_{1,n}(c)$.
\end{remark}

A Gorenstein curve of genus one may only have nodes and at most one elliptic $m$-fold point for singularities. It can always be decomposed  into a minimal elliptic subcurve (the \emph{core}) and a forest of rational trees, nodally attached to it \cite[\S3.1]{Smyth}. The core may be a smooth or nodal elliptic curve, a circle of nodally attached $\PP^1$, or an elliptic $m$-fold point. In any case, the dualising bundle of the core is trivial \cite[\S2.2]{Smyth}. Moreover, if $C$ is the elliptic $m$-fold singularity $q$, and if $B$ denotes a branch of the singularity, then $B\simeq\PP^1$ and $\omega_{C}\vert_B=\omega_B(2q)\simeq\OO_{\PP^1}$. Ampleness of the log canonical bundle can therefore be phrased combinatorially by saying that every branch of an elliptic singularity contains at least one more special point (either a node or a marking), in addition to the usual Deligne--Mumford stability for the rational trees.

\begin{remark}
    Ampleness of the log canonical bundle does not imply finiteness of the automorphism group, as opposed to the case of prestable curves.
\end{remark}
    
    A log-canonically polarised Gorenstein curve $C$ of arithmetic genus one has infinite automorphism group $\Gm$ 
    if and only if it contains an elliptic singularity each of whose branches contains exactly one special point: indeed, if $C$ is nodal then it is stable and it has finite automorphism group. If $C$ contains an elliptic singularity $q$ with branches $C_1,\ldots,C_m$, the automorphism group of $C$ is the subgroup of the automorphism group of the seminormalisation $\hat C$ at $q$ preserving the subalgebra $\OO_{C,q}\subset\OO_{\hat C,\hat q}$: these are the functions annihilated by a generic tangent vector $v=(v_1,\ldots,v_m)\in \oplus_{i=1}^m (T_{q_i}C_i\setminus0)$. The automorphism group of the seminormalisation may be infinite only if the latter is semistable, in which case $\on{Aut}(\hat C)\simeq\Gm^m$ acts on the tangent space $T_{\hat q}\hat C\simeq\oplus_{i=1}^m T_{q_i}C_i$ by the natural scaling action; since $v$ matters only up to scaling, we see that the stabiliser is just one copy of the multiplicative group.\footnote{The choice of $[v]\in\PP(T_{\hat q}\hat C)^*$ is called a \emph{crimping datum} for the singularity.}


\begin{proposition}
    $\Gcal_{1,n}$ is an 
    algebraic stack of finite type with affine diagonal. If $x$ is a geometric point of $\Gcal_{1,n}$, the stabiliser group of $x$ is either finite or isomorphic to the multiplicative group $\Gm$.
\end{proposition}
\begin{proof}
    See \cite[Proposition 2.3]{TasinViviani} and the previous discussion.
\end{proof}
A great number of compactifications of $\Mcal_{1,n}$ (in fact, all of them, as we shall see) have been identified by Bozlee, Kuo, and Neff \cite{BKN} by refining Smyth's notion of \emph{level}, which counted the number of special points on the core, to the level of partitions of the marking set $[n]$. We now delve into this combinatorics.
\begin{definition}
    Let $(C;p_1,\ldots,p_n)$ be a log-canonically polarised Gorenstein curve of arithmetic genus one, and let $E$ denote the core of $C$. We denote by $\core(C)$ the \emph{core level} of $C$, namely the partition of $[n]$ into connected components of $C\smallsetminus(E\smallsetminus\{p_1,\ldots,p_n\})$. If $C$ contains an elliptic singularity $q$, we denote by $\sing(C)$ the \emph{singularity level} of $C$, namely the partition of $[n]$ into connected components of the partial normalisation of $C$ at $q$.
\end{definition}
\begin{center}
\CoreLevelFigure
\SingularityLevelFigure
\end{center}
\begin{remark}\label{rem:sing-core-inequality}
    The set of set partitions of $[n]$, denoted by $\Pi_n$, is a poset by refinement: $P_1 \preceq P_2$ if and only if $P_2$ is a refinement of $P_1$. 
    Notice that $P\mapsto\lvert P\rvert$ (number of parts) is a rank function. 

If $C$ contains an elliptic singularity, we have $\sing(C)\preceq\core(C)$.
\end{remark}
In what follows, given points $x_i \colon\Spec(K_i)\to \Gun$ corresponding to curves $C_i$ over fields $K_i$, we say that $C_1$ specializes to $C_2$, i.e. $C_1\rightsquigarrow C_2$, if $x_2$ belongs to the closure of $x_1$ in the topological space $|\Gun|$. We will be particularly interested in isotrivial specializations: these arise from a map $\Theta \to \Gun$, with $C_1$ (respectively $C_2$) the image of the open point (respectively of the closed point) of $\Theta$. Notice that in this case, one can choose $K_1=K_2$ (the converse is also true, with some mild additional hypotheses: see \cite[Lemma 3.24]{AHLH}). The map $\Theta \to\Gun$ may only be representable (and not factor through a point) if the image of the closed point of $\Theta$ is a curve with infinite stabiliser group. It follows from the previous discussion that these are precisely the curves $C$ with $\sing(C)=\core(C)$; we call them \emph{atoms}.
\begin{center}
    \BothLevelsFigure
\end{center}
\begin{proposition}\label{prop:core-sing-specialization}
    Suppose that there is a specialisation $C_1\rightsquigarrow C_2$. Then $\core(C_1)\succeq\core(C_2)$ and $\sing(C_1)\preceq\sing(C_2)$. If the specialisation is isotrivial, 
    then at least one is an equality.
    
    In particular, if $m=\core(C_2)=\sing(C_2)$ with $k$ non-singleton parts, then 
    either
\begin{enumerate}
    \item[(i)] $\core(C_1)=m$: normalising the $k$ nodes over $\Theta$ we obtain a disjoint union of 
    an isotrivial family of $m$-marked, genus one curves, together with $k$ trivial families of stable marked rational curves, or
    \item[(ii)] $\sing(C_1)=m$: the family over $\Theta$ is equisingular, and 
    the stabilisation of its normalisation is a trivial family: what may be changing here is the crimping datum of the singularity.
\end{enumerate}
\end{proposition}
\begin{remark}
Degenerations of the first, resp. second, kind have been studied in greater generality in \cite{ChenYu}, resp. \cite{TasinViviani}.
\end{remark}
\begin{remark}
    In view of the previous formulae, we may harmlessly define $\sing(C)=-\infty$ when $C$ contains no elliptic singularity.
\end{remark}

Before delving into its proof, let us draw some of the consequences of this statement, and introduce some useful terminology in passing. We refer the reader to \cite[\S 2]{BDL6} for more details.


The behaviour of the core and singularity levels under specialisation allows us to introduce two useful stratifications of $\Gun$. We may now associate to a partition $P$ a locally closed substack $\oTail{P}$ of $\Gcal_{1,n}$ consisting of all curves of core level $P$; we call this the locus of \emph{$P$-bridges} (or $P$-tails). We have the following explicit description:
\[
 \oTail{P}=\Min{p_0}\times\prod_{i=1}^k\oM_{0,\star_i\sqcup P_i}\ ,
\]

where $p_0$ (resp. $k$) is the number of parts (of cardinality at least two) of $P$, and $\star_i$ are the attaching points of the tails.
 It follows from the proposition that: \[\cTail{P}=\bigsqcup_{P'\preceq P}\oTail{P'}\ .\]

We may also associate to a partition $P$ a locally closed substack $\oEll{P}$ of $\Gcal_{1,n}$ consisting of all curves of singularity level $P$; we call it the locus of \emph{elliptic $P$-fold points}. 
We have the following explicit description: \[
\oEll{P}=\left[\on{Tot}\left(\bigoplus_{i=1}^m \mathbb L_{0,\star_i}^{\vee}\right)_{\prod_{i=1}^m\oM_{0,\star_i\cup P_i}}\Bigg/\Gm\right],   \]
where $\mathbb L_i$ denotes the cotangent line bundle at the marking $\star_i$ (notice that when $P_i$ is a singleton $\on{Tot}(\mathbb L_{0,\star_i}^{\vee})_{\oM_{0,\star_i\cup P_i}}$ is a point). 
It follows from the proposition that: \[\cEll{P}=\bigsqcup_{P\preceq P'}\oEll{P'}\ .\]

\begin{remark}\label{rmk:close_atoms}
In particular, it follows from \Cref{rem:sing-core-inequality} and \Cref{prop:core-sing-specialization} that the substack \[\Atom{P}:=\oTail{P}\cap\oEll{P}\] consisting of elliptic $P$-bridges (and $P$-fold elliptic points) with infinite automorphism group coincides with $\cTail{P}\cap\cEll{P}$, and it is therefore closed. It can be identified with a $\Gm$-gerbe over $\prod_{i=1}^k\oM_{0,\star_i\cup P_i}$ (excluding the $i$ such that $P_i$ is a singleton).
\end{remark}
\begin{remark}\label{rmk:M0P is gms}
Let $M_{0,P}$ denote the product   $\prod_{i=1}^k\oM_{0,\star_i\sqcup P_i}$ above. It follows from \cite[Proposition 4.7 and Remark 4.8]{AlperGood} (and from the properties of $\Min{p_0}$ discovered in \cite{LekiliPolishchuk, Stevens} and summarised in \Cref{rmk:minimal}) that both $\Tail{P}\to M_{0,P}$ and $\Ell{P}\to M_{0,P}$ are good moduli space morphisms.
\end{remark}

We may now go back to proving the previous proposition. 
\begin{proof}[Proof of \Cref{prop:core-sing-specialization}]
The first statement follows from the proof of \cite[Theorem 5.2]{BKN}.

An isotrivial degeneration can be nontrivial only when the special fibre has an infinite automorphism group. So, let $x$ be a point of $\Gcal_{1,n}$ representing a $P$-bridge $C$ with infinite automorphism group $G_x\simeq\Gm$, and let $\Theta\to\Gcal_{1,n}$ be an isotrivial degeneration with special point $x$. By the local structure of quasi-separated algebraic stacks locally of finite type over a field with linearly reductive stabilizers, there is an affine \'etale chart of $\Gcal_{1,n}$ around $x$ which is stabilizer-preserving at $x$ and of the form $[W/\Gm]$ with $W$ affine \cite[Theorem 1.1]{AHR-local}. Since $\Theta$ is coherently complete along its closed point and the chart is \'etale, the map $\Theta\to\Gcal_{1,n}$ lifts to a morphism $\Theta\to [W/\Gm]$, or equivalently a $\Gm$-equivariant map $\Aaff^1\to W$.

Now, let $k$ denote the number of parts in $P$ of cardinality at least two, and let $m$ denote the number of parts in $P$; let $q$ denote the elliptic $m$-fold point of $C$, and let $\star_i,\ i=1,\ldots,k$, denote the corresponding nodes (between the core and the rational tail marked by $P_i$). Ignoring the deformations of the pointed rational tails (on which $\Gm$ acts trivially), we may identify $W$ with (a smooth cover of)
\begin{equation}
    \overline W\simeq \on{Def}_q\times\prod_{i=1}^k\on{Def}_{\star_i}.
\end{equation}
where the notation $\on{Def}_{\bullet}$ stands for the versal deformation space of the point $\bullet$ in the appropriate stack.
For any branch $i$ of the elliptic $m$-fold point $q$, observe that $\Gm$ acts with opposite weights on the tangent lines at $q$ and $\star_i$. By a result of Pinkham's \cite[Theorem 11.1]{Pinkham} (for $m\geq3$, and by hand for $m=1,2$), $\Gm$ acts with negative (depending on the convention) weights on $\on{Def}_q$. By the previous observation, $\Gm$ acts with opposite weights on $\prod\on{Def}_{\star_i}$. The $\Gm$-equivariant map $\Aaff^1\to \overline W$ must therefore factor through either $\on{Def}_q\times\{0\}$ or $\{0\}\times \prod\on{Def}_{\star_i}$, meaning that either the core level or the singularity level is preserved.
\end{proof}


Let us now introduce the compactifications from \cite[\S 7]{BKN}.

\begin{definition}
    Let $c$ be a function from the set of partitions of $[n]$ to the unit interval $[0,1]$, such that $P_1\prec P_2$ and $c_{P_2}>0$ imply $c_{P_1}=1$.\footnote{This is the only place where we slightly diverge from Bozlee, Kuo and Neff. While they further require that $c_{[n]^{\text{disc}}}=0$, we do not: we shall see that the corresponding stack has a proper good moduli space (equal to a point).} We call $c$ a \emph{BKN function}.
\end{definition}
\begin{definition}
     Let $c$ be a BKN function. The stack $\oM_{1,n}(c)\subseteq \Gun$ consists of all curves whose singularity level belongs to $Q_{\text{sing}}=\{P|c_P>0\}$, and whose core level belongs to $Q_{\text{core}}=\{P|c_P<1\}$. We call these curves \emph{$c$-stable}.
\end{definition}

Bozlee, Kuo and Neff showed that $\oM_{1,n}(c)$ is an open substack of $\Gun$ \cite[Theorem 7.6]{BKN}, hence the above definition is sufficient, and in particular it is algebraic.

The space of BKN functions $c$ is naturally endowed with the structure of a cube complex $X_{1,n}\subseteq [0,1]^{\Pi_n}$. The moduli stacks $\overline{\Mcal}_{1,n}(c)$, though, only depend on the two subsets $Q_{\text{sing}}$ (downward-closed) and $Q_{\text{curve}}$ (upward-closed) in $\Pi_n$, and in particular on the face to whose interior $c$ belongs.
In other words $\oM_{1,n}(c)$ is determined by whether $c_P$ is $0$ (allowing $P$-bridges), $1$ (allowing elliptic $P$-fold points), or strictly between the two (allowing both), not by its precise value, which we could therefore set to be $1/2$. These partitions play a special role because curves with $\sing(C)=\core(C)$ (or in other words partitions in $Q_{\text{sing}}\cap Q_{\text{curve}}$, where $c$ takes an intermediate value) are precisely those with infinite stabilizer (we call them atoms). 

\begin{remark}
Let $P_1\prec\ldots\prec P_n$ be a chain in $\Pi_n$. A BKN function  takes the following values:
\[c(P_i)=\begin {cases}
1 & \text{if }i\leq h,\\
1/2 & \text{if }h<i\leq k, \\
0 & \text{if }k<i,\end{cases}\]
for some $h,k$ such that $0\leq h\leq n$, and $h\leq k \leq h+1$. 
\end{remark}

\begin{lemma}
    The stacks $\oM_{1,n}(c)$ form an open cover of $\Gun$.
\end{lemma}
\begin{proof}
    It is enough to show the surjectivity on geometric points, so let $C$ be a Gorenstein curve of genus one, and let $P_1=\sing(C)\preceq P_2=\core(C)$. Set $c(P_1)=c(P_2)=1/2$ in case they are equal, and $c(P_1)=1,\ c(P_2)=0$ otherwise; in fact, set $c(P)=1$ for any partition strictly coarser than $P_1$, and $c(P)=0$ for any partition strictly finer than $P_1$; let $c$ be $1$ on any partition incomparable with $P_1$. This is a BKN function $c$ such that $[C]\in\oM_{1,n}(c)$.
\end{proof}
\begin{remark}\label{rmk:DM}
    If $C$ has a finite automorphism group, $c$ could be chosen to factor through $\{0,1\}$, so the corresponding $\oM_{1,n}(c)$ could be chosen to be a proper Deligne--Mumford stack.
    This construction can be generalised as follows: let $I\subseteq\Pi_n$ be a (potentially empty) collection of incomparable partitions. Set $c(P)=1/2$ for $P\in I$, $c(P)=1$ for any $P$ coarser than a partition in $I$, and $c(P)=0$ for any partition finer than a partition in $I$. Notice that this is well-defined because no partition can be at the same time finer than $P_1\in I$ and coarser than $P_2\in I$, otherwise we would have $P_1\prec P_2$, contradicting their incomparability. The remaining partitions are incomparable to $\langle I\rangle$, and we may extend $c$ to them arbitrarily as a BKN function, for instance extending by $0$ \emph{or} extending by $1$. The latter choice minimizes the Picard rank of the resulting stack $\overline{\Mcal}_{1,n}(c)$, as we will prove in the last section. 
\end{remark}

\section{Classification of proper good moduli spaces}
The main result of this section is the following.
\begin{theorem}\label{theo:classification}
    An open substack $\Mcal \subset \Gcal_{1,n}$ admits a proper good moduli space if and only if there exists a BKN function $c$ such that $\Mcal \simeq \oM_{1,n}(c)$.
\end{theorem}

\begin{proof}
    The result follows from \Cref{prop:theta complete} (cf. \Cref{gms_criterion}) and \Cref{prop:classification-of-proper-gms}.
\end{proof}
We first prove that every $\oM_{1,n}(c)$ admits a proper good moduli space, and then use this information to prove the other direction of the classification statement.
\begin{proposition}\label{prop:theta complete}
    The stack $\oM_{1,n}(c)$ is universally closed, $\Theta$-complete and $S$-complete. 
\end{proposition}
\begin{proof}
    We will show that $\oM_{1,n}(c)$ is dominated by an algebraic stack with the desired properties, and descend them by \Cref{lem:descending_closedness}, \Cref{lem:descending_S-completeness} and \Cref{prop:descending-theta-complete}.
    
    \paragraph{\textbf{Universal radii}} We recall the main construction of \cite{RSPW1}. Let $C\to S$ be a log curve of arithmetic genus one with tropicalisation $\Gamma$. This means that $C$ is a prestable curve endowed with a standard logarithmic structure over the logarithmic scheme $(S,M_S)$, and $\Gamma$ is the dual graph of (the fibres of) $C$ with edge lengths in $\overline{M}_S$. Let $\delta\colon\Gamma\to\overline{M}_S$ be the conewise-linear function ``distance from the core'', i.e. the function that takes value $0$ on the core and has slope $1$ on every edge not contained in the core (in the outgoing direction) 
    
    Heuristically, it is helpful to think about standard tropical curves by picking a monoid homomorphism $\overline{M}_S\to \RR_{\geq0}$. The distance function then becomes a conewise-linear function $\delta\colon\Gamma\to\RR_{\geq0}$. The ray $\mathbb{R}_{\geq0}$ can be subdivided along the images of the vertices of $\Gamma$, and in turn $\Gamma$ can be subdivided along the preimages of these newly introduced vertices of $\mathbb{R}_{\geq0}$, so as to ensure that $\delta$ keeps mapping cones to cones. All of these subdivisions can be performed in families of tropical curves, but they require subdividing the base as well (i.e., locally extending the monoid $\overline{M}_S$), in a process known as universal semistable reduction \cite{Molcho}. This produces a logarithmic blow-up $\Mrad$ of $\oM_{1,n}$.
    
    More precisely, $\Mrad$ represents the following stack over log schemes:
 \[ \Mrad(S)=\{C\to S\text{ log curve }| \text{ the set }\{\delta(v)|v\in V(\Gamma)\}\text{ is comparable in }\overline{M}_S\} .\]
Over $\Mrad$ we have tautological families: \[\widetilde C\to\widetilde\Gamma\xrightarrow{\delta} R.\]
 Here $R$ corresponds to the aforementioned subdivision of $\mathbb{R}_{\geq0}$, and it can be thought of as an algebraic stack (Olsson fan \cite{CavalieriChanUlirschWise, BCW}) over $\Mrad$, that is a logarithmic blowup of $\Mrad\times [\Aaff^1/\Gm]$. On the other hand, $\widetilde C$ is a semistabilisation of $C$ induced by the subdivision $\widetilde \Gamma$ of $\Gamma$. 
 
 A point $r$ of $R$ determines a disc of radius $r$ around the core of $\widetilde C$, and contracting its strict interior (either by the method of \cite[Proposition 3.7.6.1]{RSPW1} or by \cite{Boz21}) we obtain a Gorenstein contraction $\overline C$, thus defining a morphism:\footnote{One of the original motivations for \cite{RSPW1} was to construct a universal resolution of indeterminacy for the birational maps $\oM_{1,n}(m_1)\dashrightarrow\oM_{1,n}(m_2)$ between Smyth's spaces. The relevance of the distance function $\delta$ depends on Smyth's key insight that the semistable tails of a Gorenstein singularity of genus one are \emph{balanced} subcurves \cite[Definition 2.11]{Smyth}: in tropical terms, discs around the core.} \[\Psi\colon R\to\Gun.\]
 For any BKN function $c$, let $R(c)$ denote the preimage of $\oM_{1,n}(c)$ under $\Psi$, and let $\Psi_c$ denote the restriction.

 \paragraph{\textbf{Properties of $R(c)$}}
 The combinatorics of the radius is related to that of the singularity as follows: the inner valence, i.e. the number of edges rooted in the vertex of genus one, corresponds to the singularity level of $\overline C$, and the outer valence to its core level. The outer valence always refines the inner valence (cf. \Cref{rem:sing-core-inequality}), strictly if and only if the circle passes through a stable vertex of $\widetilde C$, if and only if $\overline C$ has finite automorphism group.
 \begin{figure}[h]
\begin{minipage}{0.45\textwidth}
    \begin{tikzpicture}[grow=right,
  level distance=2.5cm,
  level 1/.style={sibling distance=1.5cm},
  level 2/.style={sibling distance=1cm}]
  
  \node (root1) [circle, draw] {$1$}
    child {node[circle, fill=black, inner sep=2pt] {}
      child {node[circle, fill=black, inner sep=2pt] {}
        child {node {1}}
        child {node {2}}}
      child {node {3}}
    }
    child {node[circle, fill=black, inner sep=2pt] {}
      child {node {4}}
      child {node {5}}
    };

    \node (root2) [below=2.5cm of root1, circle, fill=black, inner sep=2pt, label=above:{$r_0$},label=below:{$0$}] {}
        child {node (r1) [circle, fill=black, inner sep=2pt,label=above:{$r_1$}, label=below:{$r_{\min}{\color{blue}=r_{\max}}$}] {}
            child {node (r2) [circle, fill=black, inner sep=2pt,label=above:{$r_2$},label=below:{${\color{red}r_{\min}=}r_{\max}$}] {}
                child {node[label=above:{$r_3$},label=below:{$+\infty$}] {}}
                }};
        \node(n1)[below =1cm of r1]{\includegraphics[width=2.5cm]{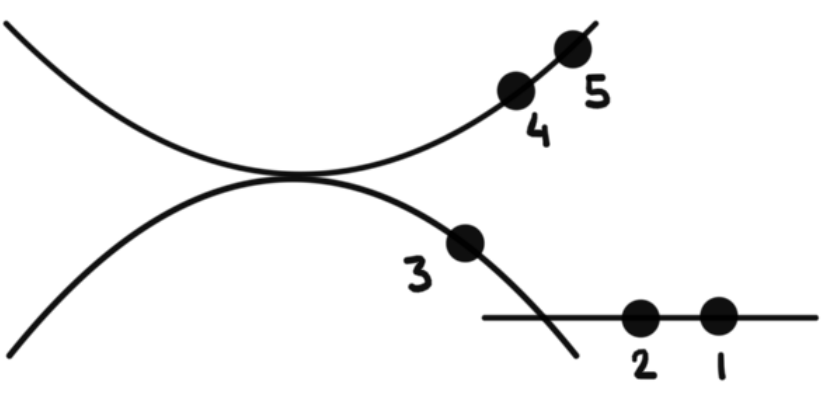}};
        \node[on grid=true, below right=3cm of n1]{\includegraphics[width=2.5cm]{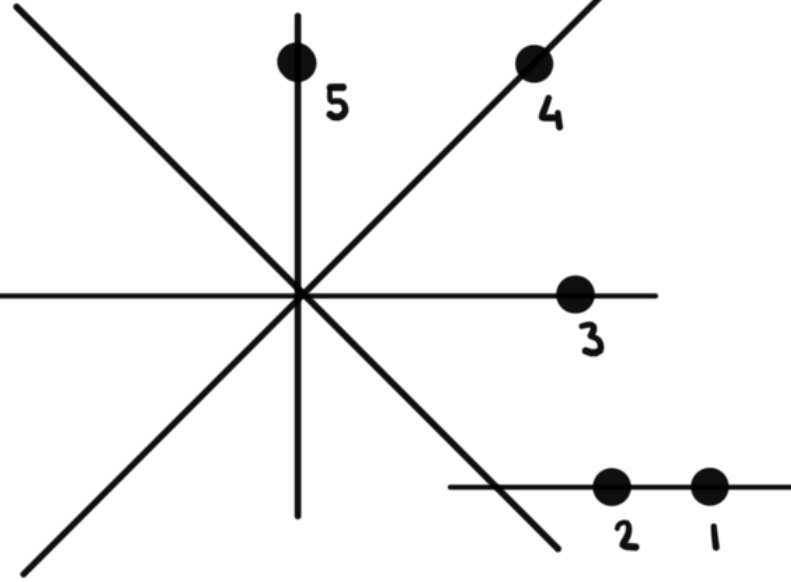}};
        \node[below =1cm of r2]{\includegraphics[width=2cm]{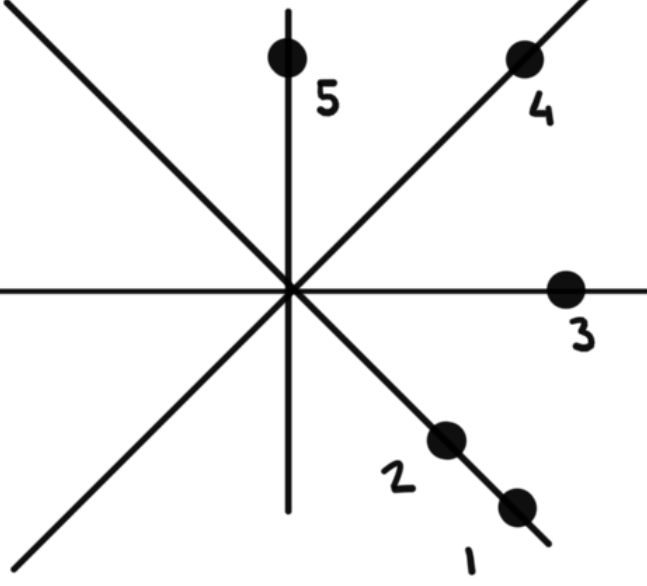}};
\end{tikzpicture}
\end{minipage}\hfill
\begin{minipage}{0.45\textwidth}
\begin{center}
\begin{tabular}{ c | c | c | c }
 $P_0$ & $P_1$ & $P_2$ & $P_3$ \\
 $[5]$ & $\{123|45\}$ & 
 $\{12|3|4|5\}$ & $\disc{[5]}$ \\\hline 
 $c_{P_0}=1$ &  $c_{P_1}=1$ & $c_{P_2}=\begin{cases}
     {\textcolor{red} 1} \\
     1/2 \\
     {\textcolor{blue} 0} \\
 \end{cases}$ & $c_{P_3}=0$
\end{tabular}
\end{center}
\end{minipage}
\caption{An example of levels and radii.}
\end{figure}

  For every cone of (the tropicalisation of) $\Mrad$, as the radius $r$ increases from $0$ to $+\infty$  we obtain a finite ordered chain of partitions $P_0=[n] \prec P_1\prec \ldots\prec P_l\prec P_{l+1}=[n]^{\text{disc}}$ recording the inner valence of the circle of radius $r$, so 
  $l$ is the number of different (positive) distances of stable vertices from the core. Conversely, every one of these partitions $P_i$ 
  determines a maximal stable radius $r_i\leq+\infty$ such that the circle of radius $r_i$ has inner valence coarser than or equal to $P_i$.
  
  The  BKN function $c$ determines two radii:
  \[r_{\min}=\max\{r_i:c_{P_i}=1\}\quad\text{and}\quad r_{\max}=\max\{r_i:c_{P_i}>0\},\]
  such that $r_{\min}\leq r_{\max}$ and they are adjacent (because of the BKN property), meaning that there is no other $r_i$ strictly between them. Notice moreover that $r_{\min}<r_{\max}$ happens exactly when $c$ takes value $1/2$ on the partition corresponding to $r_{\max}$.
  By construction, $R(c)$ is the segment $[r_{\min},r_{\max}]$ over $\Mrad$.
  
  Every partition $S$ corresponds to a divisor of $\Mrad$, whose generic point represents a curve with smooth core and $k$ rational tails, all of them at the same distance from the core, where $k$ is the number of parts of $S$ of cardinality greater than or equal to $2$. Let $c$ be a BKN function, and let $d$ denote the number of partitions on which $c$ takes value $1/2$. Recall that giving a map to $\Theta$ corresponds to picking a Cartier divisor: let $\gamma_c\colon\Mrad\to\Theta^d$ denote the map classifying the corresponding divisors.
  The above description implies that $R(c)$ fits 
  into the following diagram with a Cartesian square, cf. \cite[p. 160]{BKN}:
\bcd
\oM_{1,n}(c)& \ar[l,"\Psi_c"] R(c)\ar[r]\ar[d] & \Theta^{2d}\ar[d,"\mu"] \\
& \oM_{1,n}^{\text{rad}}\ar[r,"\gamma_c"] &\Theta^d,\\
\ecd
 where $\mu$ is induced by the multiplication map $\Aaff^2\to\Aaff^1$ (we could say that $R(c)$ is obtained from $\Mrad$ by means of a global $\overline{ST}$-construction).
 Since $\Mrad$ is a proper Deligne--Mumford stack, it is universally closed, $\Theta$- and $S$-complete. As $\mu$ is universally closed, $\Theta$- and $S$-complete, we deduce from the cartesianity of the square that $R(c) \to \Mrad$ is universally closed, $\Theta$- and $S$-complete. This implies that $R(c)$ itself is universally closed, $\Theta$- and $S$-complete.

 \paragraph{\textbf{Properties of $\Psi_c$}}
The map $\Psi_c$ is $\Theta$- and $S$-complete thanks to \Cref{lem:map-in-the-middle}; it is therefore enough to show that $R(c)$ is generically closed over $\overline{\Mcal}_{1,n}(c)$, in order to conclude that $\Psi_c$ is universally closed (\Cref{lem:generic_univ_closed}). Thus, we can test the valuative criterion for generically smooth families. 

Consider a family of $c$-stable curves $\oC$ over the spectrum $\Delta$ of a DVR, with smooth generic fiber. We may assume that the special fibre contains an elliptic $m$-fold point, otherwise there is nothing to prove. Apply semistable reduction to obtain a contraction $\rho\colon\widetilde C\to \oC$ (possibly after a ramified base-change on $\Delta$). The classifying map $\Delta\to\oM_{1,n}$ lifts to $\Mrad$ automatically, because $\Delta$ has a rank $1$, valuative logarithmic structure (in other words, the image of $\Delta\to\oM_{1,n}$ is not contained in the blow-up centre of $\Mrad\to\oM_{1,n}$). The tropicalisation of $\Delta$ is a ray mapping to the interior of some cone of $\Mrad$; let $r_{\min}$ and $r_{\max}$ be defined accordingly as above. The exceptional locus of $\rho$ is a balanced subcurve of genus one, i.e. a disc around the core; call its radius $r_\Delta$. The $c$-stability core condition guarantees that $r_\Delta\geq r_{\min}$, and the singularity condition guarantees that $r_{\Delta}\leq r_{\max}$, hence $\Delta$ lifts to $R(c)$, thus concluding the proof that $\Psi_c$ is (generically) universally closed.

Surjectivity of $\Psi_c$ can be proved in the same way: for a $c$-stable curve $\oC$ over $\Spec(k)$, pick a one-parameter smoothing over $\Delta$, and lift $\Delta$ to $R(c)$, finding in particular a preimage of $\oC$.

\paragraph{\textbf{Conclusion}} We may apply \Cref{lem:descending_closedness}, \Cref{lem:descending_S-completeness} and \Cref{prop:descending-theta-complete} to make the properties of being universally closed, $\Theta$- and $S$-complete descend along $\Psi_c$, implying that $\overline{\Mcal}_{1,n}(c)$ is universally closed, $\Theta$- and $S$-complete.
\end{proof}

We move on to proving the converse direction of \Cref{theo:classification}. First, we show that any universally closed open substack is a union of core and singularity strata.

\begin{lemma}\label{lem:union-of-strata}
    Let $\mathcal{U}\subset \Gun$ be an open substack which is universally closed. The following  hold:
    \begin{itemize}
        \item[a)] if $\mathcal{U} \cap \Atom{P}\neq \emptyset$ then $\Atom{P}\subset \Tail{P}\subset \mathcal{U}$;
        \item[b)] if $\mathcal{U} \cap \Tail{P}\neq \emptyset$ then $\Tail{P}\smallsetminus \Atom{P}\subset \mathcal{U}$.
    \end{itemize}
    Analogous statements hold with $\Ell{P}$ in place of $\Tail{P}$. If in addition $\Ucal$ is $S$-complete:
     \begin{itemize}
        \item[c)] If $\mathcal{U} \cap \Tail{P}\neq \emptyset$ and $\mathcal{U} \cap \Ell{P}\neq \emptyset$ then $\Atom{P}\subset \mathcal U$.
    \end{itemize}
\end{lemma}

\begin{proof}
    We will only prove the result for $\Tail{P}$. The case of $\Ell{P}$ is entirely analogous.
    
    \begin{itemize}[leftmargin=.5cm]
        \item[a)]\label{pf-pt-a} If $\mathcal{U} \cap \Atom{P}\neq \emptyset$, then $\mathcal{U} \cap \Atom{P} \subset \Atom{P}$ is an open substack which is universally closed, because of \Cref{rmk:close_atoms}. We know that $\Atom{P}$ is a $\Gm$-gerbe over a smooth connected projective scheme $M_{0,P}$, therefore the composition $\Ucal \cap \Atom{P}\rightarrow M_{0,P}$ is universally closed and flat, thus surjective. This implies $\Atom{P} \subset \Ucal$. We thus have: 
    $$
    \begin{tikzcd}
    \Atom{P} \arrow[r, hook] \arrow[rd] & \Ucal \cap \Tail{P} \arrow[r, hook] & \Tail{P} \arrow[ld] \\
     & M_{0,P}    &                        
    \end{tikzcd}
    $$
where the two vertical maps are good moduli space morphisms (\Cref{rmk:M0P is gms}). 
For every $m\in M_{0,P}$, the substack $\Ucal\cap(\Tail{P})$ is open in the fibre over $m$, and it contains the unique closed point $(\Atom{P})_{m}$, thus it contains the whole fibre. This implies that $\Tail{P} \subset \Ucal$.

\item[b)] First of all, the statement holds for the open substacks of the form $\oM_{1,n}(c)$ by construction. 

Suppose that $\Ucal \cap \Tail{P} \neq \emptyset$ and, because of a), we can assume $\Ucal \cap \Atom{P}=\emptyset$. 
Since $\Tail{P}$ is irreducible and $\Ucal$ is open, $\Ucal$ contains the generic point $C_\eta$ of $\Tail{P}$. Let $C_0$ be any point of $\Tail{P}\smallsetminus\Atom{P}$: we want to show that $C_0$ belongs to $\Ucal$. Fix a morphism $f\colon\Spec(R)\to \Tail{P}$ witnessing the specialisation $C_\eta\rightsquigarrow C_0$.

Since $\Ucal$ is universally closed, up to a ramified base change, we may extend the generic point $f_{|\eta}$ to a morphism $f'\colon\Spec(R)\to\Ucal\cap\cTail{P}$ with special fibre $C'_0$ in $\Ucal$. 

Since the stacks of the form $\oM_{1,n}(c)$ form an open cover of $\Gun$, and since $\Spec R$ is quasi-compact, there exists a $c$ such that $f'$ factors through $\oM_{1,n}(c)$. Thus $\oM_{1,n}(c)$ contains $\Tail{P}$ and $f$ factors through $\oM_{1,n}(c)$ as well.

If $C'_0$ has finite automorphism group, then $c$ can be chosen so that $\oM_{1,n}(c)$ is a proper Deligne--Mumford stack, hence by separatedness $C_0\simeq C_0'$ is in $\Ucal$.

Otherwise, the $S$-completeness of $\oM_{1,n}(c)$ allows us to complete $(f,f')$ to a family over $\overline{ST}_R$; let $C_{00}$ denote the special fibre. We obtain the following diagram of specialisations

\bcd
& C_0, \begin{cases} \core(C_0)=P,\\ \sing(C_0)=P'\precneq P\end{cases} \ar[dr,squiggly,"\theta"] & \\
C_{\eta}, \begin{cases} \core(C_\eta)=P,\\ \sing(C_\eta)=-\infty\end{cases}\ar[ur,squiggly,"f"]\ar[dr,squiggly,"f'"] && C_{00} \\
& C_0', \begin{cases} \core(C_0')=P''\preceq P,\\ \sing(C_0')=P''\end{cases} \ar[ur,squiggly,"\theta'"] &
\ecd
where $\theta$ and $\theta'$ are specialisations along $\Theta$.
Applying \Cref{prop:core-sing-specialization}, we deduce the following inequalities:
\[\max\{P',P''\}\preceq\sing(C_{00})\preceq\core(C_{00})\preceq P''\]
We deduce that the ones above are all equalities. In particular, the degeneration $\theta'$ is trivial, and $C_0\rightsquigarrow C_0'\in \Ucal$, hence $C_0\in\Ucal$ too.
\item[c)] Consider an atom of level $P$, call it $C_{00}$, and let $\hat C$ denote the stabilisation of its seminormalisation. By moving off from $0$ in the crimping space, we can construct an isotrivial degeneration of a curve $C_0'$ of singularity level $P$ and seminormalisation $\hat C$ to $C_{00}$. Now choose a smoothing of $C_0'$ to $C_\eta$ over a discrete valuation ring $R$. By semistable reduction and up to a ramified cover of $R$, we may replace the elliptic $P$-fold point with an elliptic $P$-bridge $C_0$. Since we may assume that there is a birational contraction $C_0\to C_0'$, we see that the rational forest within $C_0$ is again determined by $\hat C$. Now $C_0$ admits an isotrivial degeneration to $C_{00}$ by \cite[Theorem 4.10]{TasinViviani}. This construction yields an $\overline{ST}_R\setminus0$-point of $\Ucal$, together with a filling in $\Gun$. By $S$-completeness of $\Ucal$, a filling exists in $\Ucal$ already; since $\Gun$ has affine diagonal, the filling is unique, so we conclude that $[C_{00}]\in\Ucal$.
\end{itemize}
\end{proof}
The next lemma motivates the very definition of a BKN function.
\begin{lemma}\label{prop:constraints}
 Let $\Ucal\subseteq \Gun$ be an open substack which is $\Theta$-complete, $S$-complete and universally closed. If $\Ucal$ contains an elliptic $m$-fold point of singularity level $P$, then it does not contain any elliptic $P'$-bridge, for $P'\precneq P$.
\end{lemma}

In formulae, $\Ucal\cap\oEll{P}\neq\emptyset$ and $P'\precneq P$ imply $\Ucal\cap\oTail{P'}=\emptyset$.
\begin{remark}
   In particular, $\Gcal_{1,n}$ does not admit a good moduli space. To make it more concrete, let $n=2$: then the irreducible cusp admits:
  \begin{enumerate}
   \item an isotrivial degeneration to a curve with two components, a cusp and a $2$-marked $\PP^1$, glued transversally at general points, and
   \item a degeneration to the tacnode with one marking per branch.
  \end{enumerate}
  The latter admits no further degeneration; on the other hand, it is not a degeneration of the former (in such a degeneration the two markings would remain contained in a connected subcurve of genus zero). We deduce that $\Gcal_{1,2}$ is not $\Theta$-complete.
\end{remark}
\begin{proof}[Proof of \Cref{prop:constraints}]
Let $P'_1$ be a part of $P'$ obtained by joining two or more parts $P_1,P_2,\ldots$ of $P$.

Suppose that $\Ucal\cap\Tail{P'}\neq\emptyset$, and in particular $\Tail{P'}\smallsetminus\Atom{P'}\subseteq\Ucal$ by \Cref{lem:union-of-strata}. Since $\Ucal$ contains $\Ell{P}\setminus\Atom{P}$ (again by \Cref{lem:union-of-strata}), and since $\Ell{P}$ is in the closure of $\Ell{P'}$ - the open $\Ucal$ must contain $\Ell{P'}\setminus\Atom{P'}$ as well. But then, by $S$-completeness, $\Ucal$ must contain $\Atom{P'}$ as well. Now consider a (minimal) $P'$-fold point: on one hand, it admits a degeneration to a $P$-fold point; on the other, it admits an isotrivial degeneration to a $P'$-atom. Both degenerations are contained in $\Ucal$. By $\Theta$-completeness we should be able to fill in the diagram: on one hand, there is an elliptic point separating $P_1$ from $P_2$; on the other, there is a node separating $P_1'$ from the core, and this is a contradiction. In a diagram:
\bcd
C_\eta,\ \sing(C_\eta)=P'
\ar[d,squiggly,"\Theta"]\ar[r,squiggly] & C_0,\ \sing(C_0)=P\ar[d,squiggly,"\Theta"]\\
C_{\eta0},\ \sing(C_{\eta0})=\core(C_{\eta0})=P'\ar[r,squiggly]&C_{00},
\ecd

\[
P\preceq\operatorname{sing}(C_{00})
\preceq\operatorname{core}(C_{00})
\preceq P'\precneq P.
\]
\end{proof}

Thanks to the previous lemma, we may check that the following is a BKN function: 
$$ c_{\mathcal{U}}(P)=\begin{cases}
    1 \qquad {\rm if}\,\,\,\mathcal{U} \cap \Ell{P} \neq \emptyset \,\, {\rm and} \,\, \mathcal{U} \cap \Tail{P} =\emptyset , \\
     \frac{1}{2} \qquad {\rm if}\,\,\,\mathcal{U} \cap \Ell{P} \neq \emptyset \,\, {\rm and} \,\, \mathcal{U} \cap \Tail{P} \neq\emptyset , \\
    0 \qquad {\rm if}\,\,\, \mathcal{U} \cap \Ell{P} = \emptyset ;
\end{cases}
$$
or, equivalently,
$$ c_{\mathcal{U}}(P)=\begin{cases}
    1 \qquad {\rm if}\,\,\, \Ell{P} \subseteq \Ucal \,\, {\rm and} \,\, \mathcal{U} \cap \Tail{P} =\emptyset , \\
     \frac{1}{2} \qquad {\rm if}\,\,\,\Atom{P}\subseteq \Ell{P}\cup\Tail{P}\subseteq\Ucal , \\
    0 \qquad {\rm if}\,\,\, \Tail{P} \subseteq \Ucal \,\, {\rm and} \,\, \mathcal{U} \cap \Ell{P} = \emptyset .
\end{cases} $$

\begin{proposition}\label{prop:classification-of-proper-gms}
   Let $\mathcal{U}\subseteq \Gun$ be an open substack with a proper good moduli space. Then $\Ucal=\oM_{1,n}(c_{\Ucal})$.  
\end{proposition}

\begin{proof}
We prove the double containment at the level of geometric points. Let $u$ be a geometric point of $\Ucal$ and let $P:=\core (u)$. Then $\Ucal \cap \mathbf{Core}_{P}\neq \emptyset$, thus $c_{\Ucal}(P)<1$. If the curve associated to $u$ does not contain elliptic $m$-fold points, then by definition it is in $\oM_{1,n}(c_{\Ucal})$. Otherwise, suppose $\sing(u)=P'$. Since $u \in \Ucal \cap \Ell{P'}$, we have $c_{\Ucal}(P')>0$. Therefore $u \in \oM_{1,n}(c_{\Ucal})$.

For the other inclusion, let $u \in \oM_{1,n}(c_{\Ucal})$ be a geometric point. If $P:=\core(u)$, then $c_{\Ucal}(P)<1$ and thus $\Ucal \cap \Tail{P}\neq \emptyset$. By \Cref{lem:union-of-strata}, we get $\Tail{P} \smallsetminus\Atom{P}\subseteq \Ucal$. If $u$ does not have elliptic $m$-fold points, then $u \in \Tail{P} \smallsetminus\Atom{P}\subseteq \Ucal$. Otherwise, set $P':=\sing(u)$: if $P'\precneq P$, we conclude again that $u \in \Ucal$. Finally, if $P=P'$, then $u\in \Atom{P}\subseteq \Ucal$ by \Cref{lem:union-of-strata} and we win.
\end{proof}
\section{To be, or not to be projective}
In this section we discuss the line bundles on the moduli stacks $\oM_{1,n}(c)$, the conditions under which they descend to the good moduli spaces, and the projectivity of the latter. 

\subsection{Line bundles and divisors on moduli stacks and spaces}
Given a subset $S\subset [n]$, we denote by $\disc{S}$ the partition of $[n]$ given by $\{S,\{i\}_{i\in [n]\smallsetminus S}\}$. Note that $\cTail{\disc{S}}$ is the divisor of rational tails marked by $S$; we shall denote it by $\Delta_{0,S}$, and its class by $\delta_{0,S}$.
We denote by $\lambda$ the class of the Hodge line bundle $\Hcal=\pi_*\omega_{\Ccal/\Gun}$ over $\Gcal_{1,n}$, and by $\psi_i$ the conormal line bundle of the $i^{\rm th}$ universal section. Furthermore, we set $\delta_0=\sum \delta_{0,S}$, $\delta=\delta_{\rm irr}+\delta_0$, and $\psi=\sum \psi_i$.
With a little abuse of notation, we will denote by the same symbols the restrictions of the above classes to any open substack of $\Gcal_{1,n}$. 

The good moduli spaces of $\oM_{1,n}(c)$ are denoted $\cM(c)$, and the image of $\cTail{P}$ (respectively $\cEll{P}$) is denoted $\overline{\mathrm{Core}}_{P}$ (respectively $\overline{\mathrm{Sing}}_P$). Similarly, the good moduli space of the stack $\Min{m}$ is denoted $\mathrm{Min}_m$. Since we have little control on finite automorphism groups, we will mostly be concerned with the \emph{rational} Picard groups of the moduli spaces.

\begin{lemma}\label{lem:lb_on_Gun}
    The stack of pointed reduced Gorenstein curves of genus one has the following properties.
    \begin{enumerate}
        \item $\Gun$ is normal and Gorenstein, with klt singularities if and only if $n\leq 10$.
        \item The Picard group of $\Gun$ is freely generated by $\lambda$ and $\delta_{0,S}$ for $S\in[n]^n_2$.
        \item The following relations hold in the Picard group:
        \begin{align*}
            &\delta_{\rm irr}=12\lambda, && \\
            &\psi_i=\lambda+\sum_{i\in S\in [n]_2^n}\delta_{0,S}, &\psi=n\lambda+\sum_{S\in [n]_2^n}\lvert S\rvert \delta_{0,S},\\
            &K_{\Gun}=13\lambda-2\delta+\psi=(n-11)\lambda+\sum_{S\in [n]_2^n}(\lvert S\rvert-2)\delta_{0,S}\\
             &K_{\Gun}+\psi+a\lambda=(2n-11+a)\lambda+\sum_{S\in [n]_2^n}(2\lvert S\rvert-2)\delta_{0,S}&&
        \end{align*}
    \end{enumerate}
\end{lemma}
\begin{proof}

\noindent
    \begin{enumerate}
        \item This local question can be settled at the level of the miniversal deformation base for the elliptic $m$-fold singularities, for which see \cite[Corollary 1.5.15 and Proposition 1.6.3(iii)]{LekiliPolishchuk} (recall that rational is equivalent to klt under the Gorenstein assumption).
        \item Normality implies that the cycle class map $\on{Pic}(\Gun)\hookrightarrow\on{Cl}(\Gun)$ is injective. Since $\oM_{1,n}\subseteq\Gun$ has codimension two, we know that $\on{Cl}(\Gun)$ is generated by $\delta_{\rm irr}$ and $\delta_{0,S}$ (since it is true for $\overline{\Mcal}_{1,n}$); the latter are Cartier divisors by the deformation theory of nodes. The relation $\delta_{\rm irr}=12\lambda$ shows that the former can be represented by a Cartier divisor as well. See \cite[Fact 3.28]{CTV1} for an analogous statement in all genera (albeit only for lci curves).
        \item See \cite[Theorem 2.2]{ArbarelloCornalba} for the statement about stable curves.
    \end{enumerate}
\end{proof}

\begin{remark}\label{rmk:lb_on_Gun}
    The previous statements hold for $\oM_{1,n}(c)$ as well, with the difference that $\Delta_{0,S}$ is a divisor in $\oM_{1,n}(c)$ if and only if $c(\disc{S})<1$, and the klt condition becomes $c(P)=0$ for every partition with 11 or more parts.
\end{remark}


\begin{lemma}\label{lem:Pic-pgms}
  Let $P_1,\ldots,P_k\in\Pi_n$ be all the partitions such that $c(P_i)=1/2$. A $\mathbb{Q}$-line bundle $\Lcal=a\lambda+\sum_{c(\disc{S})<1}b_S\delta_{0,S}$ on $\oM_{1,n}(c)$ descends to $\cM(c)$ if and only if:
    \[a-\sum_{P_{i,j}\in P_i\colon \lvert P_{i,j}\rvert\geq2}b_{P_{i,j}}=0 \qquad \text{for all } i=1,\ldots,k.\]
    In particular:
\begin{enumerate}
    \item[(i)] The canonical bundle $K_{\oM_{1,n}(c)}$ descends if and only if $m_i+k_i=11$ for all $i$ such that $c(P_i)=1/2$, where $m_i$ (resp. $k_i$) is the number of parts (resp. parts of cardinality at least two) of $P_i$. In particular, all the coarse moduli spaces $\cM(c)$ (for $c$ integral) are $\mathbb{Q}$-Gorenstein.

   \item[(ii)] The line bundle $\psi_i$ descends if and only if 
   $$c(P)=1/2 \implies \{i\} \notin P.$$ In particular, the line bundle $\psi$ descends if every partition $P$ such that $c(P)=1/2$ does not contain singletons.
   
    \item[(iii)] The line bundle $K_{\oM_{1,n}(c)}+\psi+a\lambda$ descends to $\cM(c)$ if and only if the number of parts of all $P_i$ with $c(P_i)=1/2$ is equal to $\frac{11-a}{2}$, independent of $i$.
    \end{enumerate}
\end{lemma}
\begin{proof}

        Compare with the proof of \Cref{prop:core-sing-specialization}. The condition for a line bundle to descend is that the action of the automorphism group of a geometric point be trivial on the fibre of the line bundle at that point \cite[Theorem 10.3]{AlperGood}. As we are dealing with $\mathbb{Q}$-line bundles, we may therefore focus on a single atomic curve of level $P=P_i$. The automorphism group of an atomic curve, which can be identified with $\Gm$, acts on a miniversal deformation base of the latter. The action is trivial except for the deformation space of the elliptic $m$-fold point $q_0$ and the nodes $q_1,\ldots,q_m$ separating the core from the rational tails.
        
        The Hodge bundle comes from the miniversal deformation base of $q_0$, where it can be identified with the tangent line of any branch of the core at $q_0$; its multiples come from the same base.
        if $S \notin P$ the line bundle $\delta_{0,S}$ has weight $0$, otherwise, when $S=P_{i,j}\in P_i$, it can be identified $\Gm$-equivariantly with the tangent line to the core at $q_j$.
        In this second case $\lambda$ and $\delta_{0,S}$ have opposite weights. Notice that $\lambda$ is not trivial when restricted to the residual gerbe, thus we get the equation in the statement. 
\end{proof}




\subsection{Purity}
The following definitions are very natural in the pursuit of the maximal open substacks of $\Gun$ admitting a projective good moduli space; alas, we shall see in \Cref{impure} that not all the maximal cells of the cube complex $X_{1,n}$ correspond to pure functions.
\begin{definition}
    We say a BKN function $c$ is \emph{pure of level $m$} if $c(P)=1$ for $\lvert P \rvert<m$ and $c(P)=0$ for $\lvert P \rvert > m$.
\end{definition}
Note that all of Smyth's spaces are pure, and they are minimal with respect to inclusion among all pure open substacks. On the opposite side of the Deligne--Mumford to Artin spectrum we have the maximal pure open substacks, which will be denoted as follows:
\begin{definition}
For $m\leq n$, denote by $\Gun(m)$ the open substack of $\Gun$ where only elliptic $\ell$-fold points, $\ell\leq m$, and no elliptic $k$-bridges, $k\leq m-1$, are permitted. Its good moduli space is denoted by $\cG(m)$. In this case the BKN function $c$ is pure of level $m$ and takes the value $1/2$ on the partitions of length exactly $m$.
\end{definition}
In particular, for $m=n$ we recover the stack $\Min{m}$ and its good moduli space $\mathrm{Min}_m\simeq \Spec(\mathbb{C})$. 

We shall mostly be concerned with the projectivity of the pure moduli spaces, starting from the bottom, that is $\cG(m)$.
\begin{notation}
    Following Smyth \cite[\S 4.2]{SmythII}, we set
    \[H(s)=\psi-\delta_0-s\lambda=(n-s)\lambda + \sum_{S\in[n]_2^n} (|S|-1)\delta_{0,S}.\]
    This is a line bundle over $\Gcal_{1,n}$. By \Cref{lem:lb_on_Gun} it can be identified with $(K_{\Gun}+\psi+a\lambda)/2$ for $s=\frac{11-a}{2}$ (it therefore descends to all moduli spaces of pure level $s\in[0,n]\cap\ZZ$ by \Cref{lem:Pic-pgms}), and with $K_{\Gun}+\delta_0 + (11-s)\lambda$. It also coincides with Smyth's line bundle $D(12-s)$, where $D(t)=t\lambda+\psi-\delta$.
\end{notation}

\begin{proposition}[\cite{SmythII}, Proposition 4.10]\label{prop:nefness}
    For $m-1\leq s\leq m$, the line bundle $H(s)$ is nef on $\cM(m-1)$. Moreover, it is ample for $m-1<s < m$.
\end{proposition}
Notice that there is a diagram
\bcd
\oM_{1,n}(m-1)\ar[r,hook]\ar[d] & \Gun(m)\ar[dd] &\oM_{1,n}(m)\ar[l,hook']\ar[d] \\
\cM(m-1)\ar[dr,"\phi"]&&\cM(m)\ar[dl,"\psi"']\\
&\cG(m)&\\
\ecd
where the bottom diagonal arrows are contractions (the right one being small). It is tempting to postulate that, in the limit case $s=m$, the line bundle $H(s)$ may be ample on $\cG(m)$.
\begin{lemma}\label{lm:Pic G(m)}
    If $m=2$, the Picard rank of $\cG(m)$ is $2^{n-1}-n$. Otherwise, for $2<m<n$, it is $1$, and the rational Picard group is freely generated by $H(m)$.
\end{lemma}
\begin{proof}
    The Picard group of $\Gun(m)$ is freely generated by $\delta_{0,S}$ for $2\leq\lvert S\rvert\leq n-m+1$ and $\lambda$, so the Picard rank is $\sum_{i=2}^{n-m+1}{n\choose i}+1$. The equations that a line bundle must satisfy in order to descend are indexed by the set partitions of $[n]$ into $m$ parts. We obtain a homogeneous linear system whose associated matrix $\Xi$ is obtained by adjoining a $-1$ column to the incidence matrix with $\Pi_{n,m}$ rows and $[n]_2^{n-m+1}$ columns: for a partition $P$ and a subset $S$ as above \[\Xi_{P,S}=\begin{cases}
        1&\text{ if $S$ is a block of $P$,} \\
        0&\text{ otherwise,}
    \end{cases}\]
    and $\Xi_{P,\lambda}=-1$ for all $P$.

    \textbf{Subsets of cardinality $n-m+1$}: they may only appear in the corresponding partition $\disc{S}$, hence they contribute a block identity matrix (which we can ignore for the rest of the argument) and all the corresponding coefficients $b_S$ must be equal to $a$, the coefficient of $\lambda$.
    
    \textbf{Case $m=2$:} all remaining equations are of the form $a=b_S+b_{S^c}$. There are $2^n-2n-1$ remaining divisors, and $2^{n-1}-n-1$ remaining equations (they are  independent since they involve disjoint pairs of variables), so the Picard rank is $2^{n-1}-n$ as claimed, with a basis given by:
    \[\lambda+\sum_{S\in{[n]\choose n-1}}\delta_{0,S} + \frac{1}{2} \sum_{S \in [n]_2^{n-2}} \delta_{0,S},\qquad \left\{\delta_{0,S}-\delta_{0,S^c}\right\}_{S\in[n]_2^{n-2}}.\]

    \textbf{Case $m\geq3$.} Let $$\Lcal=a\lambda+\sum_{S\in[n]_2^{n-m+1}}b_S\delta_{0,S}$$ be a line bundle that descends.
    
   First, we claim that $b_S$ only depends on $\lvert S\rvert$. Fix $2\leq k\leq n-m+1$. Let $A$ and $B$ be two subsets of $[n]$ of cardinality $k$; assume that $C=A\cap B$ has cardinality $k-1$, so $A=C\cup\{a\}$ and $B=C\cup\{b\}$. Pick any partition $P'$ of $[n]\smallsetminus(C\cup\{a,b\})$ into $m-2$ parts (notice that this set has cardinality $n-k-1\geq m-2$). Comparing the equations for $P_1=(A|b|P')$ and $P_2=(B|a|P')$ we see that $b_A=b_B$. Since any two subsets of cardinality $k$ may be connected  through a sequence of elementary exchanges, we conclude that the coefficient $b_S$ only depends on the cardinality of $S$. Let us denote these coefficients by $b_2,\ldots,b_{n-m+1}$.
    They can be related as follows. 
    
    For any $2\leq k\leq n-m$, the partition type $(k|n-m+2-k|1^{m-2})$ gives:
    \[a=b_k+b_{n-m+2-k}.\]
    On the other hand, for $2\leq l\leq n-m-1$, the partition type $(l|n-m+1-l|2|1^{m-3})$ gives:
    \[a=b_2+b_l+b_{n-m+1-l}.\]
    Putting these together we obtain:
    \[b_{l+1}=a-b_{n-m+1-l}=b_2+b_l.\]
    We see that every coefficient is determined by the value of $b_2$. Setting this to be $1$ we obtain:
    \[b_{l+1}=l\qquad\text{and}\qquad a=n-m,\]
    so $\on{Pic}_{\mathbb Q}(\cG(m))$ is freely generated by the line bundle $H(m)$ for any $m\geq3$.
\end{proof}

\subsection{An impure, non-projective  example}\label{impure}
In this section we construct a BKN function that is not pure, and such that no non-trivial line bundle descends to the good moduli space (which, therefore, cannot be projective). Consider the following six partitions $P_1,\ldots,P_6$ of $[8]$:
\[\disc{(12|34)}, \disc{(12|56)}, \disc{(12|37)}, \disc{(12|48)}, \disc{(34|56)}, \disc{(37|48|56)}.\]
Let $c$ be the BKN function defined as follows:
\begin{itemize}
    \item $c(P_i)=1/2$ for $i=1,\ldots,6$;
    \item $c(P)=0$ for any partition $P$ finer than one of the $P_i$, and
    \item $c(P)=1$ for all other partitions.
\end{itemize}
We claim that $\on{Pic}\left(\mathbb M_{1,8}(c)\right)=0$.

Indeed, $\on{Pic}\left(\oM_{1,8}(c)\right)$ is free on six elements:
\[\lambda,\quad \delta_{0,12},\quad \delta_{0,34},\quad \delta_{0,37},\quad \delta_{0,48},\quad \delta_{0,56}.\]

By \Cref{lem:Pic-pgms}, a line bundle $\Lcal=a\lambda+\sum_{I}b_I\delta_{0,I}$ descends if and only if the coefficients satisfy the homogeneous linear system associated to the following matrix:
\[\begin{pmatrix}
    -1 & 1 & 1 & & & \\
    -1 &1 &&&& 1\\
    -1 &1& &1&&\\
    -1 &1&&&1&\\
    -1&&1&&&1\\
    -1&&&1&1&1
\end{pmatrix}.\]
This matrix is invertible, so no non-trivial line bundle descends to the good moduli space.

\subsection{Projectivity of $\cG(m)$}
The open embedding $\oM_{1,n}(m-1)\subset \Gcal_{1,n}(m)$ determines a morphism
\[ \varphi\colon \cM(m-1) \longrightarrow \mathbb{G}_{1,n}(m). \]
In this section we will prove the following: 
\begin{theorem}\label{thm:G(m) proj}
    The good moduli space $\cG(m)$ of $\mathcal{G}_{1,n}(m)$ is projective and $H(m)$ is ample on it.
\end{theorem}
We will prove this by applying the Nakai-Moishezon criterion for the ampleness of line bundles on algebraic spaces.
\begin{remark}
    We cannot resort to Kleiman's criterion directly, as the latter requires a projective scheme to start with and may fail for algebraic spaces. A direct application of \cite[Lemma 4.9]{FS-ample} as in \cite[Lemma 4.11]{SmythII} is complicated by our scarce geometric understanding of the only line bundle class $H(m)$ on $\cG(m)$, which - as opposed to Smyth's spaces - is far from being $\mathbb Q$-factorial. Nonetheless, our proof was inspired by Smyth's.

    An alternative (minimal model-theoretic) proof, along the lines of \cite[Proposition 7.3]{AFS-projectivity}, appears to work for $m<11$ only, namely when the singularities are klt (see \Cref{lem:lb_on_Gun}): consider 
    the small contraction $\psi\colon\cM(m)\to\cG(m)$. 
    The section ring of $H(m)$ is the same whether it is computed on $\cG(m)$ or on $\cM(m)$ (as can be seen by passing to the stacks). Next, by identifying $H(m)$ with $D(12-m)$ (in Smyth's notation) and applying \cite[Proposition 4.7]{SmythII}, we deduce that the section ring of $H(m)$ can be computed on the moduli space of stable curves $\oM_{1,n}$. Writing
    \[H(m)=\frac{n-m}{12}\delta_{irr}+\sum_{k=2}^n(k-1)\sum_{\lvert S\rvert=k}\delta_{0,S}\]
    we realise that $H(m)$ is $\mathfrak S_n$-invariant, and that it lies in the interior of the invariant effective cone (as described in \cite[Theorem 5.1]{ChenCoskun}), so in particular it is big. Writing
    \[H(m)=K_{\oM_{1,n}}+\frac{11-m}{12}\delta_{irr}+\delta_0\]
    for $m<11$, we may apply \cite[Corollary 1.2.1]{BCHM} to deduce the finite generation of the section ring. Now, working on $\cM(m)$ again, the line bundle $H(m)$ is big, nef, and has finitely generated section ring, so it is semiample (see \cite[Theorem 2.3.15]{LazarsfeldI}). Since it is numerically trivial exactly on the curves that are contracted under $\psi$, the line bundle $H(m)$ is semiample and positive on all curve classes on $\cG(m)$, and therefore ample.
\end{remark}
As we now need to distinguish between stacks (and spaces) of curves with core level $m$ in $\Gun(m)$ and $\oM_{1,m}(m-1)$, we set $\Tail{m}^\circ=\Tail{m}\cap \oM_{1,n}(m-1)=\Tail{m}\setminus\Atom{m}$. Let $\mathrm{Core}_m$ and $\mathrm{Core}_m^\circ$ denote their respective good moduli spaces. Observe that if $n=m$ then $\Tail{m}=\mathcal{G}_{1,m}(m)$ and $\Tail{m}^\circ=\oM_{1,m}(m-1)$.
\begin{lemma}\label{lm:Gmm}
    For $m=n$ we have $\cG(n)=\Spec(\mathbb{C})$, thus $\varphi$ is the contraction to a point.
\end{lemma}
\begin{proof}
    We have an isomorphism $\Gcal_{1,n}(n)\simeq [\Spec(A)/\Gm]$, where the induced action of $\Gm$ on each generator of $A$ has positive weight  \cite[Theorem 1.5.7]{LekiliPolishchuk}, thus $\mathbb{G}_{1,n}(n)=\Spec(A^{\Gm})=\Spec(\mathbb{C})$.
\end{proof}
\begin{lemma}\label{lm:exc in Tm}
    For $n>m$, the morphism $\varphi$ is birational and its exceptional locus is equal to $\mathrm{Core}_m^\circ$.
\end{lemma}
\begin{proof}
     By \Cref{prop:core-sing-specialization} we have $\pi^{-1}(\pi(\Tail{m}))=\Tail{m} \cup \Ell{m}$, where $\pi\colon \Gcal_{1,n}(m)\to \cG(m)$ is the good moduli space morphism. Therefore $\Ucal=\Gcal_{1,n}(m)\smallsetminus(\Tail{m} \cup \Ell{m})$ is a saturated open substack, and 
     the induced morphism $U=\pi(\Ucal)\to \cG(m)$ is an open embedding. Furthermore, the natural inclusion $\Ucal\subset\oM_{1,n}(m-1)$ also induces an open embedding $U\to \cM(m-1)$.

    We claim that $\varphi$ is an isomorphism over $U$: indeed, we have a commutative diagram
    \[
    \begin{tikzcd}
        U \ar[r] \ar[dr, "="] & \cM(m-1) \times_{\cG(m)} U \ar[r] \ar[d] & \cM(m-1) \ar[d, "\varphi"] \\
        & U \ar[r] & \cG(m)
    \end{tikzcd}
    \]
    where the bottom arrow is an open embedding, hence the second top arrow is an open embedding. As the composition of the two top arrows is also an open embedding, we deduce that the first top arrow is an open embedding. By separatedness of the vertical arrows, we deduce that the first top arrow is also closed: as $U$ is dense in $\cG(m)$, this implies that the first top arrow is an isomorphism, thus proving that $\mathrm{Exc}(\varphi)\subset \mathrm{Core}_m^{\circ}$.

    Let $P=\{P_1,\ldots, P_k, P_{k+1},\ldots, P_m\}$ be a partition with $|P_j|>1$ if and only if $j\leq k$. Recall that we set $\oM_{0,P}=\prod_{i=1}^{k}\oM_{0,P_i\cup \star_i}$ in \Cref{rmk:M0P is gms}. Observe that $\mathrm{Core}_P\simeq\oM_{0,P}$, and $\mathrm{Core}_P^\circ\simeq \overline{\mathbb{M}}_{1,m}(m-1) \times \oM_{0,P}$. The induced morphism of good moduli spaces $\mathrm{Core}_P^\circ \to \mathrm{Core}_P$ corresponds then to the projection $$\varphi_P\colon\overline{\mathbb{M}}_{1,m}(m-1) \times \oM_{0,P} \longrightarrow \oM_{0,P}.$$ As $\mathrm{Core}_m^{\circ}$ is the union of such $\mathrm{Core}_P^{\circ}$, this shows that $\mathrm{Core}_m^{\circ}\subset \mathrm{Exc}(\varphi)$.
\end{proof}
Before proving \Cref{thm:G(m) proj}, we restrict our attention to the boundary divisors $\Delta_{0,S} \subset \cG(m)\ (2\leq |S| <n-m+1)$ and the image of the exceptional divisor $\mathrm{Core}_m$.
\begin{lemma}\label{lm:case Z contained in divisor}
    Assume that $H(m)$ is ample on $\mathbb{G}_{1,r}(m)$ for $r<n$. Then the restriction of $H(m)$ to $\Delta_{0,S}\subset\cG(m)$ is ample. 
    The same holds for the restriction to $\mathrm{Core}_P,\ |P|=m$.
\end{lemma}
\begin{proof}
    We have $\Delta_{0,S}\simeq \mathbb{G}_{1,[n+1]\smallsetminus S}(m) \times \oM_{0,S\cup\{\star\}}$. Let $\mathrm{pr}_1^*\psi$ (respectively $\mathrm{pr}_2^*\psi$) be the pullback of $\sum_{i\in [n+1]\smallsetminus S} \psi_i$ (respectively $\sum_{i\in S\cup\{\star\}} \psi_i$). Furthermore, denote $\mathrm{pr}_1^*\delta$ (respectively $\mathrm{pr}_2^*\mathcal{D}$) the pullback of $\sum \delta_{0,T}$ (respectively $\sum\mathcal{D}_T$, the sum of the boundary divisors in $\oM_{0,S\cup\{\star\}}$). We have that $\lambda|_{\Delta_{0,S}}=\rm{pr}_1^*\lambda$ and $\psi-\delta_0|_{\Delta_{0,S}}=\rm{pr}_1^*(\psi-\delta_0)+\rm{pr}_2^*(\psi-\mathcal{D})$ \cite[page 1877, first and second equations]{SmythII}. We deduce
    \begin{align*}
        H(m)|_{\Delta_{0,S}}&=(\psi-\delta_0-m\lambda)|_{\Delta_{0,S}} \\
        &=\mathrm{pr}_1^*(\psi - \delta_0 - m\lambda) + \mathrm{pr}_2^*(\psi - \mathcal{D}) = \mathrm{pr}_1^*(H(m)) + \mathrm{pr}_2^*(\psi-\mathcal{D})
    \end{align*}
    Our assumption tells us that $\mathrm{pr}_1^*(H(m))$ is the pullback of an ample line bundle on the first factor, and $\mathrm{pr}_2^*(\psi-\mathcal{D})$ is ample on $\oM_{0,S\cup\{\star\}}$ \cite[Lemma 4.8]{SmythII}, thus their sum is ample on the product.

    For $|P|=m$, recall that $\mathrm{Core}_P=\prod_{P_i\in P} \oM_{0,P_i\cup\{\star_i\}}$. The restriction of $H(m)$ to $\mathrm{Core}_P$ is then equal to the sum of the ample divisors $\mathrm{pr}_i^*(\psi-\mathcal{D})$, thus it is ample.
\end{proof}
\begin{proof}[Proof of \Cref{thm:G(m) proj}]
    We prove that $H(m)$ is ample on $\cG(m)$ by induction on the number of marked points. The base case is $\mathbb{G}_{1,1}(1)=\Spec(\mathbb{C})$, which is obvious. More generally, the case $m=n$ is handled by \Cref{lm:Gmm}, so we may assume $m<n$.
    We want to employ the Nakai-Moishezon criterion, i.e. we want to show that for every subvariety $Z\subset \cG(m)$ of dimension $d>0$, we have $(H(m)^d\cdot [Z])>0$. 
    

    By \Cref{lm:case Z contained in divisor} the restriction of $H(m)$ to $\Delta_{0,S}$ and $\mathrm{Core}(P)$ is ample, so we may assume that $Z$ is not contained in either substack of $\cG(m)$. 

    Let $\widetilde{Z}$ be the strict transform of $Z$ in $\cM(m-1)$, so that by the projection formula we have 
    \[(H(m)^d\cdot [\widetilde{Z}])=(H(m)^d\cdot \varphi_*[\widetilde{Z}])=(H(m)^d\cdot [Z]).\]
    Moreover, the intersection of $\widetilde{Z}$ with the boundary divisor $\Delta_{0,S}\subset \cM(m-1)$ is either proper or empty.

    Write $\widetilde{H}(t)=\frac{1}{n-t}H(t)$. As rescaling $H(m)$ by a positive value does not affect the intersection number, we can equivalently prove that $(\widetilde{H}(m)^d \cdot Z)>0$. For $0<\epsilon\ll 1$ consider the difference
    \begin{align*}
        \widetilde{H}(m) - \widetilde{H}(m-\epsilon) = &\frac{\epsilon}{(n-m)(n-m+\epsilon)}\sum (|S|-1)\delta_{0,S}=\sum c_s \delta_{0,S}
    \end{align*}
    with $c_S> 0$. We can write
    \begin{align}\label{eq:Htilde}
        (\widetilde{H}(m)^d\cdot [\widetilde{Z}]) = &(\widetilde{H}(m)^{d-1}\cdot\widetilde{H}(m-\epsilon)\cdot [\widetilde{Z}])+(\widetilde{H}(m)^{d-1}\cdot (\widetilde{H}(m)-\widetilde{H}(m-\epsilon))\cdot [\widetilde{Z}])
    \end{align}
    We will prove the positivity of this expression by induction on $d$. First assume that $d=1$. The line bundle $\widetilde{H}(m-\epsilon)$ is ample on $\cM(m-1)$ by \Cref{prop:nefness}, hence  $(\widetilde{H}(m-\epsilon)\cdot [\widetilde{Z}])>0$; furthermore our assumption ensures that $(\delta_{0,S}\cdot [\widetilde{Z}]) \geq 0$, and we conclude that \eqref{eq:Htilde} is strictly positive.
    
    Suppose now $d>1$. Ampleness of $\widetilde{H}(m-\epsilon)$ implies that (up to taking a multiple of all line bundles involved) we can represent $\widetilde{H}(m-\epsilon)\cdot \widetilde{Z}$ as the class of a closed subscheme $\widetilde{W} \subset \cM(m-1)$ of dimension $d-1$ that is not fully contained in the exceptional locus $\mathrm{Core}_m$, so that $\varphi_*[\widetilde{W}]=[W]$ with $W\subset \cG(m)$ of dimension $d-1$. Thus:
    \begin{align*}
        (\widetilde{H}(m)^{d-1}\cdot\widetilde{H}(m-\epsilon)\cdot [\widetilde{Z}]) = &(\widetilde{H}(m)^{d-1}\cdot [\widetilde{W}])=(\widetilde{H}(m)^{d-1}\cdot [W])>0
    \end{align*} 
    using the projection formula and the inductive assumption.
    We also have
    \begin{align*}
        (\widetilde{H}(m)^{d-1}\cdot (\widetilde{H}(m)-\widetilde{H}(m-\epsilon))\cdot [\widetilde{Z}] )&= \sum c_{S}(\widetilde{H}(m)^{d-1}\cdot [\Delta_{0,S}]\cdot[\widetilde{Z}]) \\
        &=\sum c'_S(\widetilde{H}(m)^{d-1}\cdot [\Delta_{0,S}\cap \widetilde{Z}])
    \end{align*}
    where the $c'_S\geq 0$ are obtained by multiplying $c_S$ by the intersection multiplicity of $\Delta_{0,S}\cap Z$. As $\widetilde{H}(m)$ is nef by \Cref{prop:nefness} we deduce that for each $S$ we have $(\widetilde{H}(m)^{d-1}\cdot [\Delta_{0,S}\cap \widetilde{Z}])\geq 0$. Again, we conclude that \eqref{eq:Htilde} is strictly positive.
\end{proof}
\subsection{Projectivity of $\cM(c)$ in the pure case}
We now study the projectivity of $\overline{\mathbb{M}}_{1,n}(c)$ for $c$ a pure BKN function. We start by recalling the notion of relative semistable loci \cite[Definition 11.1]{AlperGood} when the base is a good moduli space.
\begin{definition}
Let $\mathcal{X}$ be a stack with a line bundle $\mathcal{L}$ and a good moduli space $p\colon \mathcal{X} \to X$. The \emph{relative semistable locus} $\mathcal{X}(\mathcal{L})^{ss}_X$ consists of those points $x$ such that there exists an open neighborhood $V$ of $p(x)$, an integer $d>0$ and a section $s$ of $\mathcal{L}^{d}|_{p^{-1}(V)}$ such that $s(x)\neq 0$.
\end{definition}
\begin{proposition}[\cite{AlperGood}, Theorem 11.5]\label{prop:gms of sstable locus}
    The relative semistable locus $\mathcal{U}=\mathcal{X}(\mathcal{L})^{ss}_X$ has a good moduli space $U$ and the open embedding $\mathcal{U} \hookrightarrow \mathcal{X}$ induces a quasi-projective morphism $f\colon U\to X$. Moreover, the line bundle $\mathcal{L}|_{\mathcal{U}}$ descends to an $f$-ample $\mathbb Q$-line bundle on $U$.
\end{proposition}
Analogously to what happens in geometric invariant theory, relative semistability of points can be checked by studying maps from $\Theta$. First recall that the weight of a line bundle $\Lcal=[\mathbb{A}^1/\Gm]$ on $B\Gm$ is the weight of the $\Gm$-action on $\mathbb{A}^1$. For a morphism $f\colon\Theta\to\Xcal$ inducing a non-trivial homomorphism
\[\Gm\longrightarrow\on{Aut}(f(0))\simeq\Gm,\qquad t\longmapsto t^r\]
and a line bundle $\Lcal$ on $\Xcal$ with $f^*\Lcal_{|B\Gm}$ having weight $\ell$, define the normalised weight $\hat{\on{w}}(f^*\Lcal|_{B\Gm})$ of $\Lcal$ along $f$ as  $\frac{\ell}{\lvert r \rvert}$. This is independent of taking root stacks of $\Theta$ and of changing the identification $\on{Aut}(f(0))\simeq\Gm$.

\begin{lemma}[{\cite[Lemma 2.7]{DLIstable}}]\label{lm:semistability criterion}
    A point $x$ is relatively $\mathcal{L}$-semistable for a good moduli space $\mathcal{X} \to X$ if and only if for every representable morphism $f\colon\Theta\to \mathcal{X}$ that maps the generic point of $\Theta$ to $x$ the weight of $f^*\mathcal{L}|_{B\Gm}$ along $f|_{B\Gm}$ is $\geq 0$\footnote{There is a sign mistake in the original statement contained in the cited preprint.}. 
\end{lemma}
Recall \Cref{prop:core-sing-specialization}: any representable map $\Theta\to \Gcal_{1,n}$ sends the closed point to $\Atom{P}$, and it must factor either
(i) through $\Tail{P}$, or (ii) through $\Ell{P}$. In particular, setting $\lvert P\rvert=m$, both 
\begin{enumerate}
    \item[(i)] $\Tail{P}\simeq\Gcal_{1,m}(m)\times \oM_{0,P}$, and
    \item[(ii)] $\Ell{P}\simeq\left[\on{Tot}\left(\bigoplus_{i=1}^m \mathbb L_{0,\star_i}^{\vee}\right)_{\oM_{0,P}}\Bigg/\Gm\right]$
\end{enumerate} admit a projection to the variety $\oM_{0,P}$, so in any case the map from $\Theta$ will factor through a fibre of these projections. 

\begin{lemma}\label{lm:weights}
    Let $f\colon \Theta\to \mathcal{G}_{1,n}$ be as above and assume that it does not factor through $\Atom{P}$. Then
    \[
    \hat{\on{w}}(f^*\lambda|_{B\Gm})= \left\{
    \begin{matrix}
        1 &\text{if }f\text{ is of type (i)}\\ 
       -1 &\text{if }f\text{ is of type (ii)}
    \end{matrix}\right., \quad 
    \hat{\on{w}}(f^*\delta_{0,B}|_{B\Gm})=\left\{ \begin{matrix}
        -1 &\text{if }f\text{ is of type (i), }B\in P\\ 
        1 &\text{if }f\text{ is of type (ii), } B\in P\\
        0 &\text{otherwise}
    \end{matrix}\right.
    \]
\end{lemma}
\begin{proof}
Case (i): denote by $\lambda^{(m)}$ the pullback of the Hodge bundle from $\Gcal_{1,m}(m)$, and $\psi_{\star_i}$, $\mathcal{D}^{P_i}_{T_i}$ the classes pulled back from $\oM_{0,P}$. Given the embedding $j\colon \Tail{P} \to \mathcal{G}_{1,n}$ we have 
\[
j^*\lambda=\lambda^{(m)},\quad j^*\delta_{0,B}=\left\{
\begin{matrix}
-\lambda^{(m)}-\psi_{\star_i} & \text{if }B=P_i, \\
\mathcal{D}_{B}^{P_i} & \text{if }B\subsetneq P_i, \\
0 & \text{otherwise}
\end{matrix}
\right.
\]
The map $f\colon\Theta\to \mathcal{G}_{1,n}$ factors through a fiber of the projection $\Tail{P} \to \prod_i \overline{\mathcal{M}}_{0,P_i\cup\star_i}$, which can be identified with $\Gcal_{1,m}(m)$. Hence from the previous formula we deduce that
\[
f^*\lambda=\lambda^{(m)}|_{\Theta}, \quad f^*\delta_{0,B}=\left\{
\begin{matrix}
-\lambda^{(m)}|_{\Theta} & \text{if }B=P_i, \\
0 & \text{otherwise}
\end{matrix}
\right.
\]
where $\lambda^{(m)}$ is as before the Hodge line bundle on $\Gcal_{1,n}(m)$. As $\Gcal_{1,m}(m)\simeq[\Spec(A)/\Gm]$ and $\lambda^{(m)}$ restricts to the tautological line bundle on $B\Gm$ 
\cite[Theorem 1.5.7]{LekiliPolishchuk}, we deduce that its weight at the origin is $1$. 

Case (ii) is similar: the map $f\colon\Theta\to\Gun$ factors through a fibre of $\Ell{P}\to\prod_i \overline{\mathcal{M}}_{0,P_i\cup\star_i}$, which can be identified with $[\on{Tot}(\oplus_i \mathbb{L}_{0,\star_i}^{\vee})/\Gm]$. 
It follows then from Smyth's computations \cite[Example 3.5]{SmythII} that $f^*\lambda$ has weight $-1$ and $f^*\delta_{0,B}$ has weight equal to $1$ when $B=P_i$ and equal to zero when $B\not\in P$. 
\end{proof}
\begin{proposition}\label{prop:inequalities}
Suppose that $c$ is a pure BKN function. Then $\oM_{1,n}(c) \subset \Gcal_{1,n}(m)$ is a relative $\mathcal{L}$-semistable locus for the line bundle $\mathcal{L}=a\lambda + \sum_{S} b_S \delta_{0,S}$ if and only if there exist values $a$, $\{b_S\}_{S\subset [n]}$ such that for every partition of the form $P=\{P_1,\ldots, P_m\}$ we have
    \begin{equation}\label{eq:inequalities}
    \tag{$\star$}
    \left\{
    \begin{matrix}
        a - \sum_{\lvert P_i \rvert \geq 2} b_{P_i}<0 & \text{ if }c(P)=1,\\
        a - \sum_{\lvert P_i \rvert \geq 2} b_{P_i}=0 & \text{ if }c(P) = \frac{1}{2}, \\
        a - \sum_{\lvert P_i \rvert \geq 2} b_{P_i}>0 & \text{ if }c(P)= 0.
    \end{matrix}
    \right.
    \end{equation}
\end{proposition}
\begin{proof}
    We apply \Cref{lm:semistability criterion}. Any $f\colon\Theta \to \Gcal_{1,n}(m)$ that generically factors through $\oM_{1,n}(c)$ must map $B\Gm$ to a curve in $\Atom{P}$ with $\lvert P \rvert=m$. We know all the possible such maps thanks to \Cref{prop:core-sing-specialization}:
    \begin{itemize}
        \item we have $c(P)=1$ if and only if $\hat{\on{w}}(f^*\mathcal{L}|_{B\Gm})\geq 0$ when $f$ factors through $\Ell{P}$, and $\hat{\on{w}}(f^*\mathcal{L}|_{B\Gm})< 0$ when $f$ factors through $\Tail{P}$; 
        \item we have $c(P)=0$ if and only if $\hat{\on{w}}(f^*\mathcal{L}|_{B\Gm})< 0$ when $f$ factors through $\Ell{P}$, and $\hat{\on{w}}(f^*\mathcal{L}|_{B\Gm})\geq 0$ when $f$ factors through $\Tail{P}$;
        \item we have $c(P)=\frac{1}{2}$ if and only if $\hat{\on{w}}(f^*\mathcal{L}|_{B\Gm})\geq 0$ in both cases.
    \end{itemize}
    \Cref{lm:weights} translates these conditions into the claimed (in)equalities on the coefficients of $\mathcal{L}$.
\end{proof}
\begin{remark}
    This provides an explicit description of the semistability (or, what is the same, $\mathbb{Q}$-factorialisation) fan of $\cG(m)$, see \cite[\S 2]{GoriThesis} where the theory is developed.
    
    Given a pure BKN function $c$ and a minimal partition $P$ such that $c(P)=0$, it is always possible to produce BKN functions $c_0$ and $c_+$ by changing the value of $c$ at $P$ to $1/2$, resp. $1$. Locally around $\Tail{P}$, this operation produces the standard local model of a divisorial contraction (if $P=\disc{S}$) or a flip (if $P$ has more than one part with more than one element) \cite[\S 1]{ReidFlip}. Yet, the (global) conditions for the existence of a line bundle $\Lcal$ making these into minimal model theoretic $\Lcal$-flips are stricter: roughly speaking, turning in a $P$-bridge in exchange for a $P$-fold elliptic point forces us to do the same to other partitions $P'$, \emph{if} we insist on staying projective. Alternatively, switching one $P$ at a time will produce a (plethora of) non-projective model(s).
\end{remark}

\begin{corollary}\label{cor:star implies proj}
    Let $c$ be a pure BKN function satisfying the conditions \eqref{eq:inequalities} of \Cref{prop:inequalities}. Then $\overline{\mathbb{M}}_{1,n}(c)$ is projective.
\end{corollary}
\begin{proof}
   By \Cref{prop:gms of sstable locus} we have that $\overline{\mathbb{M}}_{1,n}(c)$ is projective over $\cG(m)$, which is projective by \Cref{thm:G(m) proj}.
\end{proof}
\begin{corollary}\label{cor:projectivity of big contractions}
    Let $c$ be a pure BKN function of level $m$. Assume that $c$ is constant over every partition $P$ of length $m$ such that $P\neq \disc{S}$. Then $\overline{\mathbb{M}}_{1,n}(c)$ is projective.
\end{corollary}
\begin{proof}
    Consider the line bundle $\mathcal{L}=H(m) + \sum b_{S} \delta_{0,S}$ where
    \begin{itemize}
        \item if $|S|=n-m+1$ we pick $b_S>0$ when $c(\disc{S})=1$, $b_S=0$ when $c(\disc{S})=\frac{1}{2}$ and $b_S<0$ when $c(\disc{S})=0$;
        \item if $|S|< n-m+1$, given any partition $P\neq \disc{S}$ of length $m$, we pick $b_S>0$ when $c(P)=1$, $b_S=0$ when $c(P)=\frac{1}{2}$ and $b_S<0$ when $c(P)=0$.
    \end{itemize} 
    This choice of coefficients satisfies \eqref{eq:inequalities}.
\end{proof}
\begin{theorem}\label{thm:condition for proj}
    Let $c$ be a pure BKN function. The stack $\oM_{1,n}(c)$ has a projective good moduli space if and only if $c$ satisfies \eqref{eq:inequalities}.
\end{theorem}
\begin{proof}
    One implication is \Cref{cor:star implies proj}.
    Assume that $\oM_{1,n}(c)$ has a projective good moduli space. Define the auxiliary pure BKN function $c'$, having the same weight $m$ as $c$, by setting $c'(P)=c(P)$ if $P=\disc{S}$ and $\frac{1}{2}$ otherwise. By \Cref{cor:projectivity of big contractions} the good moduli space $\overline{\mathbb{M}}_{1,n}(c')$ is projective. 
    
    Consider the open embedding $\oM_{1,n}(c)\subset \oM_{1,n}(c')$. By construction, its complement has codimension $\geq 2$, and the induced morphism $\overline{\mathbb{M}}_{1,n}(c) \to \overline{\mathbb{M}}_{1,n}(c')$ is projective. By \Cref{rmk:lb_on_Gun} the stack $\oM_{1,n}(c')$ is normal and all the line bundles on $\oM_{1,n}(c)$ extend to line bundles on $\oM_{1,n}(c')$; we can thus apply \cite[Lemma 2.18]{DLIstable} to deduce that $\oM_{1,n}(c)$ is a relative semistable locus for $\oM_{1,n}(c') \to \overline{\mathbb{M}}_{1,n}(c')$. By \cite[Proposition 2.16]{DLIstable} or via a direct computation this implies that $\oM_{1,n}(c)$ is a relative semistable locus for $\mathcal{G}_{1,n}(m) \to \cG(m)$.
\end{proof}
\begin{remark}
    The proof above shows that it is sufficient to verify \eqref{eq:inequalities} for $P\neq \disc{S}$, as we can always find coefficients $b_{S}$, $\lvert S \rvert=n-m+1$, such that the (in)equalities indexed by $P=\disc{S}$ are satisfied, and the values of such $b_{S}$ do not affect the (non-)existence of a solution to the rest of the (in)equalities.
\end{remark}

\begin{corollary}\label{cor:number of non-prog gms}
    The following hold true:
    \begin{enumerate}[leftmargin=.5cm]
        \item for $m=1,2,n-1,n$ all the pure stacks $\oM_{1,n}(c) \subset \mathcal{G}_{1,n}(m)$ have a projective good moduli space;
        \item otherwise, there are stacks $\oM_{1,n}(c)$, both Deligne-Mumford and not, whose moduli spaces are not projective.
    \end{enumerate}
    Moreover, in the second case, there are at least $\sum_{i=d+1}^{S(n,m)} \binom{S(n,m)}{i}$ Deligne-Mumford stacks with a non-projective moduli space, where $d= \sum_{i=2}^{n-m+1} \binom{n}{i}$ and $S(n,m)$ is the number of partitions of $n$ of rank $m$ (Stirling number of the second kind). 
\end{corollary}
\begin{remark}
    The result for $m=2$ should be compared with \cite{CTV1,CTV2}, where the second step (first flip) in the Hassett-Keel program is decomposed into a series of elementary flips according to the minimal model program: the projectivity of good moduli spaces of curves with at worst tacnodal singularities holds in all genera.
\end{remark}
\begin{proof}
Consider the $\mathbb{Q}$-vector space $V=\Pic_{\mathbb{Q}}(\mathcal{G}_{1,n}(m))$ generated by $\lambda$ and the classes $\delta_{0,S}$, with coordinates $a$ and $b_S$, for $2\leq |S|\leq n-m+1$. This vector space has dimension $d+1=\sum_{i=2}^{n-m+1}\binom{n}{i}+1$. For a partition $P$ of length $m$, let us denote by $H_P \subset V$ the affine hyperplane defined by the equation $\sum_{S\in P, |S|\geq 2} b_{S}=a$. 
The number of such hyperplanes is equal to $S(n,m)$. 
\Cref{thm:condition for proj} ties together the geometry of the central arrangement of the hyperplanes $\{H_P\}$ with the projectivity of the moduli spaces $\overline{\mathbb{M}}_{1,n}(c)$.

\emph{Case }$m=1,2,n-1,n$: for these values $\overline{\mathbb{M}}_{1,n}(c)$ is always projective. The cases $m=1$, $n$ are immediate; for $m=n-1$ we have $P=\disc{S}$ for every partition, thus this follows from \Cref{cor:projectivity of big contractions}. If $m=2$ the proof of \Cref{lm:Pic G(m)} shows that the $H_P$ are linearly independent, thus every BKN function $c$ of pure weight $2$ satisfies \eqref{eq:inequalities}.

\emph{Case } $m=3,\ldots,n-2 \; (n\geq5)$: the proof of \Cref{lm:Pic G(m)} shows that the $H_P$ span the subspace of $V^{\vee}$ of forms vanishing at $H(m)$ and that they are not linearly independent. This implies that there are stacks $\oM_{1,n}(c)$, both Deligne-Mumford and not, having a non-projective moduli space. In fact, the existence of linear relations among the equations defining the $H_P$ forces the existence of non-projective good moduli spaces of stacks with strictly semistable points. 

Finally, the number of regions in which the hyperplanes $H_P$ subdivide $V$ is bounded above by  the sum $\sum_{i=0}^{\dim V} \binom{S(n,m)}{i}$ (see for instance \cite[\S 2.2]{Stanley}), while the number of BKN functions of pure weight $m$ that do not assume the value $\frac{1}{2}$ is $2^{S(n,m)}$. 

Observe that we can embed the set of regions into the set of pure BKN functions with values $\neq \frac{1}{2}$ by associating to each region $R$ the pure BKN function whose value at $P$ is $1$ if $H_P>0$ in $R$, and $0$ otherwise. The image of this embedding consists of the pure BKN functions that satisfy \eqref{eq:inequalities}.
This implies that there are at least $N(n,m)=\sum_{i=d+1}^{S(n,m)} \binom{S(n,m)}{i}$ choices of pure BKN functions of weight $m$ such that $\overline{\mathbb{M}}_{1,n}(c)$ is a non-projective coarse moduli space of a Deligne-Mumford stack.
\end{proof}
\begin{corollary}\label{cor:asymptotics}
    Fix $m\geq 3$. Let $MS(n)$ be the number of Deligne-Mumford stacks $\oM_{1,n}(c) \subset \mathcal{G}_{1,n}(m)$, and let $PMS(n)$ be the number of those that have a projective coarse moduli space. Then
    \[\lim_{n\to \infty} \frac{PMS(n)}{MS(n)} = 0 \]
\end{corollary}
\begin{proof}
The denominator is the number of integral BKN functions of pure weight $m$, that is $2^{S(n,m)}$. Now let us estimate the numerator.

Set $D=\dim(V)$ and $S=S(n,m)$. Recall that $D\leq2^n$ and that $S(n,m)\sim \frac{m^n}{m!}$ for $n\to \infty$, so $D=o(S)$ and in particular $D\leq S/2$ for large $n$.

The number of regions in which the arrangement $\{H_P\}$ subdivides $V$ is less or equal to $\sum_{i=0}^{D} \binom{S}{i}$, which is bounded above by $(D+1)\binom{S}{D}$, and so (by a standard upper bound on binomial coefficients) by: \[(D+1)\left(\frac{eS}{D}\right)^D.\]

Now, taking logarithms, we get:
\[\log\frac{PMS(n)}{MS(n)}\leq \log(D+1)+D(1+\log S-\log D)-S\log 2\]

Bounding $D$ by $2^n$ and $S$ by $m^n$, we obtain:
\[\log\frac{PMS(n)}{MS(n)}\leq \log(2^n+1)+2^n(1+n\log m)-S\log 2.\]

The first two (positive) terms can be estimated using $n2^n$ (up to a constant depending on $m$); similarly, the third (negative) term can be estimated using $-\frac{\log 2}{m!}m^n$ up to a term of order $(m-1)^n$ (times a constant depending on $m$). Since both $n2^n$ and $(m-1)^n$ are $o(m^n)$ for fixed $m\geq3$ and large $n$, the right-hand side above behaves like \[-\frac{\log 2}{m!}m^n\to -\infty \qquad\text{for }n\to+\infty.\]

\end{proof}

\bibliographystyle{alpha}
\bibliography{Bibliography.bib}

@book{Alper-book,
    author = {Alper, Jarod} ,
    title = {Stacks and moduli},
    publisher = {available online at \url{https://sites.math.washington.edu/~jarod/moduli.pdf}},
    year = {2026}
}

@article{AHR-local,
    AUTHOR = {Alper, Jarod and Hall, Jack and Rydh, David},
     TITLE = {The \'etale local structure of algebraic stacks},
   JOURNAL = {Ann. Sci. \'Ec. Norm. Sup\'er. (4)},
  FJOURNAL = {Annales Scientifiques de l'\'Ecole Normale Sup\'erieure.
              Quatri\`eme S\'erie},
    VOLUME = {59},
      YEAR = {2026},
    NUMBER = {1},
     PAGES = {125--198},
      ISSN = {0012-9593,1873-2151},
   MRCLASS = {14A20},
  MRNUMBER = {5068605},
}

@article{DLIstable,
  author  = {Di Lorenzo, Andrea and Inchiostro, Giovanni},
  title   = {Stable maps to quotient stacks with a properly stable point},
  journal = {arXiv},
  year    = {2024},
  doi     = {10.48550/arxiv.2411.16141}
}

@article {AFS-projectivity,
    AUTHOR = {Alper, Jarod and Fedorchuk, Maksym and Smyth, David Ishii},
     TITLE = {Second flip in the {H}assett-{K}eel program: projectivity},
   JOURNAL = {Int. Math. Res. Not. IMRN},
  FJOURNAL = {International Mathematics Research Notices. IMRN},
      YEAR = {2017},
    NUMBER = {24},
     PAGES = {7375--7419},
      ISSN = {1073-7928,1687-0247},
   MRCLASS = {14H10 (14D23)},
  MRNUMBER = {3802125},
MRREVIEWER = {Scott\ R.\ Nollet},
       DOI = {10.1093/imrn/rnw216},
       URL = {https://doi.org/10.1093/imrn/rnw216},
}

@article {BCHM,
    AUTHOR = {Birkar, Caucher and Cascini, Paolo and Hacon, Christopher D.
              and McKernan, James},
     TITLE = {Existence of minimal models for varieties of log general type},
   JOURNAL = {J. Amer. Math. Soc.},
  FJOURNAL = {Journal of the American Mathematical Society},
    VOLUME = {23},
      YEAR = {2010},
    NUMBER = {2},
     PAGES = {405--468},
      ISSN = {0894-0347,1088-6834},
   MRCLASS = {14E30 (14E05)},
  MRNUMBER = {2601039},
MRREVIEWER = {Mark\ Gross},
       DOI = {10.1090/S0894-0347-09-00649-3},
       URL = {https://doi.org/10.1090/S0894-0347-09-00649-3},
}

@book {LazarsfeldI,
    AUTHOR = {Lazarsfeld, Robert},
     TITLE = {Positivity in algebraic geometry. {I}},
    SERIES = {Ergebnisse der Mathematik und ihrer Grenzgebiete. 3. Folge. A
              Series of Modern Surveys in Mathematics [Results in
              Mathematics and Related Areas. 3rd Series. A Series of Modern
              Surveys in Mathematics]},
    VOLUME = {48},
      NOTE = {Classical setting: line bundles and linear series},
 PUBLISHER = {Springer-Verlag, Berlin},
      YEAR = {2004},
     PAGES = {xviii+387},
      ISBN = {3-540-22533-1},
   MRCLASS = {14-02 (14C20)},
  MRNUMBER = {2095471},
MRREVIEWER = {Mihnea\ Popa},
       DOI = {10.1007/978-3-642-18808-4},
       URL = {https://doi.org/10.1007/978-3-642-18808-4},
}

@article {FS-ample,
    AUTHOR = {Fedorchuk, Maksym and Smyth, David Ishii},
     TITLE = {Ample divisors on moduli spaces of pointed rational curves},
   JOURNAL = {J. Algebraic Geom.},
  FJOURNAL = {Journal of Algebraic Geometry},
    VOLUME = {20},
      YEAR = {2011},
    NUMBER = {4},
     PAGES = {599--629},
      ISSN = {1056-3911,1534-7486},
   MRCLASS = {14H10 (14C20 14D20 14N10)},
  MRNUMBER = {2819671},
MRREVIEWER = {Flaminio\ Flamini},
       DOI = {10.1090/S1056-3911-2011-00547-X},
       URL = {https://doi.org/10.1090/S1056-3911-2011-00547-X},
}

@article{BCW,
      title={Vector bundles on {O}lsson fans}, 
      author={Luca Battistella and Francesca Carocci and Jonathan Wise},
      journal = {arXiv},
      year={2026},
      eprint={2602.05685},
      archivePrefix={arXiv},
      primaryClass={math.AG},
      url={https://arxiv.org/abs/2602.05685}, 
}

@ARTICLE{BDL6,
    AUTHOR = {Battistella, Luca and Di Lorenzo, Andrea},
     TITLE = {Integral {C}how rings of modular compactifications of {$\mathcal
              M_{1,n\leq6}$}},
   JOURNAL = {J. Reine Angew. Math.},
  FJOURNAL = {Journal f\"ur die Reine und Angewandte Mathematik. [Crelle's
              Journal]},
    VOLUME = {837},
      YEAR = {2026},
     PAGES = {289--326},
      ISSN = {0075-4102,1435-5345},
   MRCLASS = {14H10 (14D23)},
  MRNUMBER = {5110503},
       DOI = {10.1515/crelle-2026-0039},
       URL = {https://doi.org/10.1515/crelle-2026-0039},
}

@article {LekiliPolishchuk,
    AUTHOR = {Lekili, Yank\i{} and Polishchuk, Alexander},
     TITLE = {A modular compactification of {$\mathcal{M}_{1,n}$} from
              {$A_\infty$}-structures},
   JOURNAL = {J. Reine Angew. Math.},
  FJOURNAL = {Journal f\"ur die Reine und Angewandte Mathematik. [Crelle's
              Journal]},
    VOLUME = {755},
      YEAR = {2019},
     PAGES = {151--189},
      ISSN = {0075-4102,1435-5345},
   MRCLASS = {14F08 (14H10)},
  MRNUMBER = {4015231},
MRREVIEWER = {Abdelmoubine\ Amar\ Henni},
       DOI = {10.1515/crelle-2017-0015},
       URL = {https://doi.org/10.1515/crelle-2017-0015},
}

@article {Stevens,
    AUTHOR = {Stevens, Jan},
     TITLE = {The versal deformation of elliptic {$m$}-fold point curve
              singularities},
   JOURNAL = {Res. Math. Sci.},
  FJOURNAL = {Research in the Mathematical Sciences},
    VOLUME = {12},
      YEAR = {2025},
    NUMBER = {4},
     PAGES = {Paper No. 66, 16},
      ISSN = {2522-0144,2197-9847},
   MRCLASS = {14B05 (14H10 32S05)},
  MRNUMBER = {4954650},
MRREVIEWER = {Ragni\ Piene},
       DOI = {10.1007/s40687-025-00553-3},
       URL = {https://doi.org/10.1007/s40687-025-00553-3},
}

@article {CTV1,
    AUTHOR = {Codogni, Giulio and Tasin, Luca and Viviani, Filippo},
     TITLE = {On the first steps of the minimal model program for the moduli
              space of stable pointed curves},
   JOURNAL = {J. Inst. Math. Jussieu},
  FJOURNAL = {Journal of the Institute of Mathematics of Jussieu. JIMJ.
              Journal de l'Institut de Math\'ematiques de Jussieu},
    VOLUME = {22},
      YEAR = {2023},
    NUMBER = {1},
     PAGES = {145--211},
      ISSN = {1474-7480,1475-3030},
   MRCLASS = {14H10 (14D23 14E30 14H20)},
  MRNUMBER = {4556931},
MRREVIEWER = {Dawei\ Chen},
       DOI = {10.1017/S1474748021000116},
       URL = {https://doi.org/10.1017/S1474748021000116},
}

@article {ArbarelloCornalba,
    AUTHOR = {Arbarello, Enrico and Cornalba, Maurizio},
     TITLE = {Calculating cohomology groups of moduli spaces of curves via
              algebraic geometry},
   JOURNAL = {Inst. Hautes \'Etudes Sci. Publ. Math.},
  FJOURNAL = {Institut des Hautes \'Etudes Scientifiques. Publications
              Math\'ematiques},
    NUMBER = {88},
      YEAR = {1998},
     PAGES = {97--127},
      ISSN = {0073-8301,1618-1913},
   MRCLASS = {14H10 (14C30)},
  MRNUMBER = {1733327},
MRREVIEWER = {Elham\ Izadi},
       URL = {http://www.numdam.org/item?id=PMIHES_1998__88__97_0},
}

@article {CTV2,
    AUTHOR = {Codogni, Giulio and Tasin, Luca and Viviani, Filippo},
     TITLE = {On some modular contractions of the moduli space of stable
              pointed curves},
   JOURNAL = {Algebra Number Theory},
  FJOURNAL = {Algebra \& Number Theory},
    VOLUME = {15},
      YEAR = {2021},
    NUMBER = {5},
     PAGES = {1245--1281},
      ISSN = {1937-0652,1944-7833},
   MRCLASS = {14H10 (14D22 14D23 14E30)},
  MRNUMBER = {4283103},
MRREVIEWER = {Scott\ R.\ Nollet},
       DOI = {10.2140/ant.2021.15.1245},
       URL = {https://doi.org/10.2140/ant.2021.15.1245},
}

@article {BDL4,
    AUTHOR = {Battistella, Luca and Di Lorenzo, Andrea},
     TITLE = {Wall-crossing integral chow rings of {$\oM_{1,n\leq4}$}},
   JOURNAL = {Forum Math. Sigma},
  FJOURNAL = {Forum of Mathematics. Sigma},
    VOLUME = {13},
      YEAR = {2025},
     PAGES = {Paper No. e196, 28},
      ISSN = {2050-5094},
   MRCLASS = {14H10 (14C15 14D23)},
  MRNUMBER = {4997998},
       DOI = {10.1017/fms.2025.10143},
       URL = {https://doi.org/10.1017/fms.2025.10143},
}

@article {Luna,
    AUTHOR = {Alper, Jarod and Hall, Jack and Rydh, David},
     TITLE = {A {L}una \'etale slice theorem for algebraic stacks},
   JOURNAL = {Ann. of Math. (2)},
  FJOURNAL = {Annals of Mathematics. Second Series},
    VOLUME = {191},
      YEAR = {2020},
    NUMBER = {3},
     PAGES = {675--738},
      ISSN = {0003-486X,1939-8980},
   MRCLASS = {14D23 (14A20 14F20 14L24 14L30)},
  MRNUMBER = {4088350},
MRREVIEWER = {Martin\ Ulirsch},
       DOI = {10.4007/annals.2020.191.3.1},
       URL = {https://doi.org/10.4007/annals.2020.191.3.1},
}

@book {Pinkham,
    AUTHOR = {Pinkham, Henry C.},
     TITLE = {Deformations of algebraic varieties with {$\Gm$} action},
    SERIES = {Ast\'erisque},
    VOLUME = {No. 20},
 PUBLISHER = {Soci\'et\'e{} Math\'ematique de France, Paris},
      YEAR = {1974},
     PAGES = {i+131},
   MRCLASS = {14D15 (14B05)},
  MRNUMBER = {376672},
MRREVIEWER = {O.\ Riemenschneider},
}

@article {AlperGood,
    AUTHOR = {Alper, Jarod},
     TITLE = {Good moduli spaces for {A}rtin stacks},
   JOURNAL = {Ann. Inst. Fourier (Grenoble)},
  FJOURNAL = {Universit\'e{} de Grenoble. Annales de l'Institut Fourier},
    VOLUME = {63},
      YEAR = {2013},
    NUMBER = {6},
     PAGES = {2349--2402},
      ISSN = {0373-0956,1777-5310},
   MRCLASS = {14D23 (14L24 14L30)},
  MRNUMBER = {3237451},
MRREVIEWER = {Arvid\ Siqveland},
       DOI = {10.5802/aif.2833},
       URL = {https://doi.org/10.5802/aif.2833},
}

@article {AHLH,
    AUTHOR = {Alper, Jarod and Halpern-Leistner, Daniel and Heinloth,
              Jochen},
     TITLE = {Existence of moduli spaces for algebraic stacks},
   JOURNAL = {Invent. Math.},
  FJOURNAL = {Inventiones Mathematicae},
    VOLUME = {234},
      YEAR = {2023},
    NUMBER = {3},
     PAGES = {949--1038},
      ISSN = {0020-9910,1432-1297},
   MRCLASS = {14D23 (14A20)},
  MRNUMBER = {4665776},
       DOI = {10.1007/s00222-023-01214-4},
       URL = {https://doi.org/10.1007/s00222-023-01214-4},
}

@article {ChenCoskun,
    AUTHOR = {Chen, Dawei and Coskun, Izzet},
     TITLE = {Extremal effective divisors on {$\overline{M}_{1,n}$}},
   JOURNAL = {Math. Ann.},
  FJOURNAL = {Mathematische Annalen},
    VOLUME = {359},
      YEAR = {2014},
    NUMBER = {3-4},
     PAGES = {891--908},
      ISSN = {0025-5831,1432-1807},
   MRCLASS = {14H10 (14C20 14E30 14L30)},
  MRNUMBER = {3231020},
MRREVIEWER = {Flaminio\ Flamini},
       DOI = {10.1007/s00208-014-1027-5},
       URL = {https://doi.org/10.1007/s00208-014-1027-5},
}

@misc{GoriThesis,
	Author = {Gori, D.},
	Note = {Ph{D} thesis, {S}apienza {U}niversit{\`a} di {R}oma},
	Title = {Alternative compactifications of $\mathrm{M}_{g,n}$ via cluster algebras and their birational geometry},
	Year = 2025
	}

@article {Molcho,
    AUTHOR = {Molcho, S.},
     TITLE = {Universal stacky semistable reduction},
   JOURNAL = {Isr. J. Math.},
  FJOURNAL = {Israel Journal of Mathematics},
    VOLUME = {},
      YEAR = {2021},
    NUMBER = {},
     PAGES = {},
      ISSN = {},
       DOI = {10.1007/s11856-021-2118-0}
}

@article {Smyth,
    AUTHOR = {Smyth, David Ishii},
     TITLE = {Modular compactifications of the space of pointed elliptic
              curves {I}},
   JOURNAL = {Compos. Math.},
  FJOURNAL = {Compositio Mathematica},
    VOLUME = {147},
      YEAR = {2011},
    NUMBER = {3},
     PAGES = {877--913},
      ISSN = {0010-437X},
   MRCLASS = {14D23 (14H10)},
  MRNUMBER = {2801404},
MRREVIEWER = {Concettina Galati},
       DOI = {10.1112/S0010437X10005014},
       URL = {https://doi.org/10.1112/S0010437X10005014},
}

@article{SmythII, title={Modular compactifications of the space of pointed elliptic curves II}, volume={147}, DOI={10.1112/S0010437X11005549}, number={6}, journal={Compositio Mathematica}, author={Smyth, David Ishii}, year={2011}, pages={1843–1884}}

@ARTICLE{ChenYu,
       author = {{Chen}, Dawei and {Yu}, Fei},
        title = "{Gorenstein singularities with $\mathbb{G}_m$-action and moduli spaces of holomorphic differentials}",
      journal = {arXiv e-prints},
         year = 2025,
        month = jul,
          eid = {arXiv:2507.09078},
        pages = {arXiv:2507.09078},
          doi = {10.48550/arXiv.2507.09078},
archivePrefix = {arXiv},
       eprint = {2507.09078},
 primaryClass = {math.AG},
       adsurl = {https://ui.adsabs.harvard.edu/abs/2025arXiv250709078C}
}

@ARTICLE{TasinViviani,
       author = {{Tasin}, Luca and {Viviani}, Filippo},
        title = "{The {H}assett-{K}eel program via $\Theta$-stability}",
      journal = {arXiv e-prints},
         year = 2026,
        month = aug,
          eid = {arXiv:2608.14099},
        pages = {arXiv:2608.14099},
          doi = {10.48550/arXiv.2608.14099},
archivePrefix = {arXiv},
       eprint = {2608.14099},
 primaryClass = {math.AG},
       adsurl = {https://ui.adsabs.harvard.edu/abs/2026arXiv260814099T}
}

@misc{stacks-project,
    shorthand    = {Stacks},
    author       = {The {Stacks Project Authors}},
    title        = {\textit{Stacks Project}},
    howpublished = {\url{https://stacks.math.columbia.edu}},
    year         = {2018},
  }

@misc{ReidFlip,
    author       = {Reid, M.},
    title        = {\textit{What is a flip?}},
    howpublished = {\url{https://mreid.warwick.ac.uk/3folds/what_flip.pdf}},
    year         = {1992},
  }

@article {Boz21,
    AUTHOR = {Bozlee, Sebastian},
     TITLE = {Contractions of subcurves of families of log curves},
   JOURNAL = {Comm. Algebra},
  FJOURNAL = {Communications in Algebra},
    VOLUME = {49},
      YEAR = {2021},
    NUMBER = {11},
     PAGES = {4616--4660},
      ISSN = {0092-7872},
   MRCLASS = {14H10 (14T20)},
  MRNUMBER = {4304662},
       DOI = {10.1080/00927872.2021.1925903},
       URL = {https://doi.org/10.1080/00927872.2021.1925903},
}

@article {BKN,
    AUTHOR = {Bozlee, Sebastian and Kuo, Bob and Neff, Adrian},
     TITLE = {A classification of modular compactifications of the space of
              pointed elliptic curves by {G}orenstein curves},
   JOURNAL = {Algebra Number Theory},
  FJOURNAL = {Algebra \& Number Theory},
    VOLUME = {17},
      YEAR = {2023},
    NUMBER = {1},
     PAGES = {127--163},
      ISSN = {1937-0652},
   MRCLASS = {14D23 (14H10)},
  MRNUMBER = {4549282},
       DOI = {10.2140/ant.2023.17.127},
       URL = {https://doi.org/10.2140/ant.2023.17.127},
}

@article {CavalieriChanUlirschWise,
    AUTHOR = {Cavalieri, R. and Chan, M. and Ulirsch, M. and
              Wise, J.},
     TITLE = {A moduli stack of tropical curves},
   JOURNAL = {Forum Math. Sigma},
  FJOURNAL = {Forum of Mathematics. Sigma},
    VOLUME = {8},
      YEAR = {2020},
     PAGES = {Paper No. e23, 93},
   MRCLASS = {14T20 (14D23)},
  MRNUMBER = {4091085},
MRREVIEWER = {Hannah Markwig},
       DOI = {10.1017/fms.2020.16},
       URL = {https://doi.org/10.1017/fms.2020.16},
}

@article {RSPW1,
    AUTHOR = {Ranganathan, Dhruv and Santos-Parker, Keli and Wise, Jonathan},
     TITLE = {Moduli of stable maps in genus one and logarithmic geometry,
              {I}},
   JOURNAL = {Geom. Topol.},
  FJOURNAL = {Geometry \& Topology},
    VOLUME = {23},
      YEAR = {2019},
    NUMBER = {7},
     PAGES = {3315--3366},
      ISSN = {1465-3060},
   MRCLASS = {14N35 (14D23)},
  MRNUMBER = {4046967},
       DOI = {10.2140/gt.2019.23.3315},
       URL = {https://doi.org/10.2140/gt.2019.23.3315},
}

@article{instability,
      title={{On the structure of instability in moduli theory}}, 
      author={Halpern-Leistner, D.},
      year={2022},
      eprint={1411.0627},
      archivePrefix={arXiv},
      primaryClass={math.AG},
      url={https://arxiv.org/abs/1411.0627},
      journal={arXiv},
}

@incollection {Stanley,
    AUTHOR = {Stanley, Richard P.},
     TITLE = {An introduction to hyperplane arrangements},
 BOOKTITLE = {Geometric combinatorics},
    SERIES = {IAS/Park City Math. Ser.},
    VOLUME = {13},
     PAGES = {389--496},
 PUBLISHER = {Amer. Math. Soc., Providence, RI},
      YEAR = {2007},
      ISBN = {978-0-8218-3736-8; 0-8218-3736-2},
   MRCLASS = {52C35 (05B35 55R80)},
  MRNUMBER = {2383131},
       DOI = {10.1090/pcms/013/08},
       URL = {https://doi.org/10.1090/pcms/013/08},
}

@article{BreVis,
      title={{An arithmetic valuative criterion for proper maps of tame algebraic stacks}}, 
      author={Bresciani, G. and Vistoli, A.},
      year={2024},
      journal={Manuscipta Math.},
      volume={173},
      pages={1061-1071},
      doi={https://doi.org/10.1007/s00229-023-01491-6}
}


\subsection*{Funding} L.B.: this material is partly based upon work supported by the National Science Foundation under Grant No. DMS-1928930, while I was in residence at the Simons Laufer Mathematical Sciences Institute in Berkeley, California, during the Spring 2026 semester. I also benefitted from a visit to the University of Washington. I am a member of the INdAM group GNSAGA.

\end{document}